%% file: main.tex
\documentclass[american]{scrartcl}
\usepackage{babel}
\usepackage[utf8]{inputenc}

\usepackage{enumitem}

\usepackage{amssymb,amsmath,amsthm,bm,mathtools} 
\usepackage{mleftright}

\usepackage{graphicx,subcaption}
\usepackage{xcolor}
\usepackage{pgfplots}
\pgfplotsset{compat=1.16}
\graphicspath{{matlab/}}

\usepackage{csquotes}
\usepackage[style=ext-numeric, maxbibnames=9, giveninits=true
  ]{biblatex}
\DeclareFieldFormat[article,inbook,incollection,inproceedings]{titlecase:title}{\MakeSentenceCase{#1}} 
\AtEveryBibitem{\clearfield{url}}
\AtEveryBibitem{\clearfield{month}}
\renewbibmacro{in:}{%
  \ifentrytype{article}{}{\printtext{\bibstring{in}\printunit{\intitlepunct}}}}

\usepackage[english,pdfusetitle,colorlinks]{hyperref}
\usepackage[english,compress]{cleveref}

\newtheorem{thm}{Theorem}[section]
\newtheorem{cor}[thm]{Corollary}
\newtheorem{lem}[thm]{Lemma}
\newtheorem{prop}[thm]{Proposition}

\theoremstyle{definition}
\newtheorem{defin}[thm]{Definition}
\Crefname{defin}{Definition}{Definitions}
\Crefname{rem}{Remark}{Remarks}
\Crefname{thm}{Theorem}{Theorems}
\Crefname{cor}{Corollary}{Corollaries}
\Crefname{prop}{Proposition}{Propositions}
\crefname{defin}{def.}{def.}
\crefformat{defin}{def.~#2#1#3}
\Crefformat{defin}{Definition~#2#1#3}
\Crefname{lem}{Lemma}{Lemmas}
\newtheorem{rem}[thm]{Remark}
\newtheorem*{rem*}{Remark}

\newcommand\restr[2]{{
    \left.\kern-\nulldelimiterspace 
    #1 
    \vphantom{\big|} 
    \right|_{#2} 
}}

\usepackage{verbatim}

\newcommand{\N}{\ensuremath{\mathbb{N}}}
\newcommand{\NN}{\ensuremath{\mathbb{N}_0}}
\newcommand{\T}{\ensuremath{\mathbb{T}}}
\renewcommand{\S}{\ensuremath{\mathbb{S}}}
\newcommand{\M}{\ensuremath{\mathbb{M}}}
\newcommand{\Z}{\ensuremath{\mathbb{Z}}}

\newcommand{\R}{\ensuremath{\mathbb{R}}}
\newcommand{\B}{\ensuremath{\mathbb{B}}}
\newcommand{\C}{\ensuremath{\mathbb{C}}}
\newcommand{\D}{\ensuremath{\mathbb{D}}}

\newcommand{\SO}{\ensuremath{\mathrm{SO}(3)}}

\newcommand{\abs}[1]{\ensuremath{\left\vert#1\right\vert}}
\newcommand{\inn}[1]{\ensuremath{\left\langle#1\right\rangle}}

\newcommand{\e}{\mathrm{e}}
\newcommand{\im}{\mathrm{i}}
\renewcommand{\i}{\mathrm{i}}

\renewcommand{\d}{\, \mathrm{d}}

\newcommand{\norm}[1]{\left\lVert #1
\right\rVert}
\newcommand{\pare}[1]{\left( #1
\right)}

\DeclareMathOperator{\sgn}{sgn}
\newcommand*\centermathcell[1]{\omit\hfil$\displaystyle#1$\hfil\ignorespaces}

\hypersetup{
	pdftitle={A double Fourier sphere method for d-dimensional manifolds},
    pdfauthor={
    	Sophie Mildenberger, mildenberger@tu-berlin.de; 
    	Michael Quellmalz, quellmalz@math.tu-berlin.de
    }
}

\begin{document}
    
\title{A double Fourier sphere method for $d$-dimensional manifolds}
\date{}
\author{
	Sophie Mildenberger\thanks{TU Berlin, Institute of
    Mathematics, MA 4-3, Straße des 17. Juni 136, D-10623 Berlin, Germany. Orcid: \href{https://orcid.org/0000-0002-2509-3114}{[0000-0002-2509-3114]}, \href{https://orcid.org/0000-0001-6206-5705}{[0000-0001-6206-5705]}}\\
	{\footnotesize\href{mailto:mildenberger@tu-berlin.de}{mildenberger@tu-berlin.de}}
	\and 
	Michael Quellmalz\footnotemark[1]\\
	{\footnotesize\href{mailto:quellmalz@math.tu-berlin.de}{quellmalz@math.tu-berlin.de}
}
}
    
\maketitle

\input{abstract}
\input{01_introduction}
\input{02_hoelder_definitions}
\input{03_DFS_method}
\input{04_smoothness}
\input{05_series_expansion}
\input{06_examples}
\input{acknowledgments}
\input{appendix}

\emergencystretch=1em
\printbibliography[heading=bibintoc]
\end{document}

%% file: abstract.tex
\begin{abstract}
	The double Fourier sphere (DFS) method uses a clever trick to transform a function defined on the unit sphere to the torus and subsequently approximate it by a Fourier series, which can be evaluated efficiently via fast Fourier transforms. Similar approaches have emerged for approximation problems on the disk, the ball, and the cylinder. In this paper, we introduce a generalized DFS method applicable to various manifolds, including all the above-mentioned cases and many more, such as the rotation group. This approach consists in transforming a function defined on a manifold to the torus of the same dimension. We show that the Fourier series of the transformed function can be transferred back to the manifold, where it converges uniformly to the original function. In particular, we obtain analytic convergence rates in case of Hölder-continuous functions on the manifold.

  \medskip
  \noindent
  \textit{Math Subject Classifications.}
  41A65, 
  42B05,  
  42C10,  
  65T50  
  
  
\end{abstract}

%% file: 01_introduction.tex
\section{Introduction}

\paragraph{Approximation on manifolds}
Throughout the mathematical literature, there is considerable interest in approximating functions defined on manifolds, cf.\ e.g., \cite{Dem21,FusWri12,Mai20,SobAizLev21}. The problem of representing and numerically manipulating such functions arises in various areas as application including weather prediction \cite{Che00,Coi11,Mer73}, protein docking \cite{PadKazZer16}, active fluids in biology \cite{BouSloTow21}, geosciences \cite{Ger14,HiPoQu18}, and astrophysics \cite{God97,Pri81}.

One method to tackle this problem is atlas-based; smooth charts relate local patches on the manifold to  Euclidean space, where well-known approximation theory can be applied. The local approximations are subsequently combined to a global approximation. In case of overlapping patches, one is faced with the issue of suitably blending them. Doing so might require the solution of non-linear equations \cite{SobAizLev21}, incorporating tangent space projections \cite{DavSch07,Dem21}, or partition-of-unity methods \cite{MajCir17}. With non-overlapping patches, e.g.\ in computer-aided design, non-linear smoothness conditions or complicated patch-stitching methods \cite{CheCheSub21} might be required.

Another approach to approximating functions on manifolds is the ambient approximation method, cf.\ e.g., \cite{Mai20,Oda16}. Here, a function on an embedded submanifold is extended into some subset of the ambient space, usually a tubular neighborhood of the manifold \cite[§~3.1]{Mai20}. The extended function can then be approximated in $\R^d$ and restricted to the submanifold again. By Whitney's embedding theorem \cite[thm.~6.15]{Lee12}, any manifold can be embedded into $\R^d$ for sufficiently large $d$, thus this approach is applicable to general manifolds. However, the ambient approximation method entails the solution of an approximation problem in a space of possibly higher dimension.

One advantage of both of the previously mentioned methods is their generality; they are applicable to any manifold, even if there is little information on the underlying geometry, see, e.g., \cite{Oda16,SobAizLev21}. However, often the manifold of interest is actually well-known and relatively simple. In such situations, approximation bases distinctly tailored to the manifold can be used instead. As a prominent example, the spherical harmonics \cite{Mic13} are well-suited for various applications
and allow for an efficient computation \cite{DriHea94,KunPot03,Moh99}. However, these algorithms show difficulties in the numerical evaluation of associated Legendre functions, cf.\ \cite{Sch13}, and their performance does not reach that of the fast Fourier transform (FFT) on the torus, cf.\ \cite[ch.~5~\&~7]{PloPotSteTas18}.

\paragraph{DFS methods}
The \emph{double Fourier sphere} (DFS) method represents functions defined on the sphere by transforming them to a torus and subsequently considering the two-dimensional Fourier series of the transformed functions. Thus, one can take advantage of the efficiency of the FFT to approximate spherical functions. The classical DFS method was originated in 1972 by Merilees \cite{Mer73} and found various applications since, e.g., \cite{Boy78,Che00,DraWri20,ForMer97,MonNak18,Ors74,PotVan17,TowWilWri16,Yee80,YinLiFen22}. Recently, we have shown analytic approximation properties of the classical DFS method \cite{MilQue22}.
Further DFS methods have been invented for other geometries such as the disk \cite{WilTowWri17}, the cylinder \cite{ForTow20}, the ball \cite{BouTow20}, and two-dimensional axisymmetric surfaces \cite{MilForMurGreShv22}. The software \emph{Chebfun} \cite{Chebfun} uses DFS methods for computing with functions and solving differential equations on the sphere, the disk, and the ball.

In this paper, we introduce a unified approach to all these DFS methods and show its analytic convergence properties.
To this end, we define a generalized DFS method that contains as special cases all the above-mentioned manifolds.
Starting with a function $f\colon \M\to\C$ on a $d$-dimensional submanifold $\M$ of $\R^{d'}$,
the fundamental idea is to apply a coordinate transformation $\phi\colon\T^d\to\M$ with certain smoothness and symmetry properties to obtain a ($2\pi$-periodic) function $\tilde f = f\circ\phi $ defined on the $d$-dimensional torus 
\begin{equation*}
	\T^d\coloneqq \R^d /(2 \pi \Z^d).
\end{equation*} 
In particular,
we construct a one-to-one connection between $f$ and the transformed function $\tilde f \colon \T^d\to\C$, which admits certain symmetries and can be represented by a $d$-dimensional Fourier series. 

Due to the imposed symmetry properties of $\phi$, we can relate linear combinations of the Fourier basis $\e^{\im \inn{\bm{n},\cdot}}$, $\bm{n}\in\Z^d$, on the torus $\T^d$ to an orthogonal basis on the manifold~$\M$. Thus, we can approximate the original function $f$ by a series expansion on $\M$ that is based on the Fourier series of $\tilde f$ on the torus. 
Therefore, the numerical computation and evaluation of this series expansion can be performed efficiently by employing the FFT. 

We prove that our generalized DFS method transfers certain smoothness classes on the manifold $\M$ to the respective ones on the torus. Accordingly, the DFS series representation of a Hölder-continuous function on $\M$ exhibits convergence rates comparable to those of Fourier series of functions in the corresponding Hölder space on the torus $\T^d$.
We derive explicit upper bounds on the respective constants, depending on the smoothness class of the function $f$, the dimension $d$ of the manifold $\M$, and the dimension $d'$ of its ambient space.

The generalized DFS method combines aspects of the three approximation approaches mentioned above: While not being atlas-based, it does depend on transforming functions from a manifold to a subset of the Euclidean space, where the well-known theory of Fourier series can be applied. 
Our proof of the approximation rates relies on an extension of $f$ into the ambient space $\R^{d'}$, 
but the DFS method itself does not require the construction of such extensions.
To avoid the complications of combining various patches, the method instead depends on a transformation that does not necessarily capture the topology of the manifold properly. As a consequence, the basis functions on the manifold might be non-smooth on a set of measure zero,
such as the poles of the sphere. However, the method still ensures fast uniform convergence of the respective series expansion when used to approximate sufficiently smooth functions.

\paragraph{Outline of this paper} 
In \Cref{sec:hoel}, we define Hölder spaces and related function spaces on the torus and embedded submanifolds. In \Cref{sec:dfs}, we introduce the generalized DFS method and present some of its basic properties. \Cref{sec:smoothness} is dedicated to proving that the generalized DFS method preserves Hölder spaces and to providing upper bounds on the corresponding semi-norms. In \Cref{sec:series}, we develop the series approximation of functions on a manifold via the DFS method. \Cref{sec:Fourier} is concerned with the Fourier series of DFS functions and the corresponding series on the manifold, for which we need to incorporate considerations on the symmetry properties of DFS functions.
In \Cref{sec:convergence}, we study sufficient conditions for the absolute convergence of the previously introduced series and show bounds on the speed of convergence. \Cref{sec:examples} considers various manifolds to which the DFS method can be applied. Besides the well-known cases of the disk, the sphere, the cylinder, and the ball,
we also consider the rotation group, higher-dimensional spheres and balls, as well as products of manifolds that themselves admit  DFS methods. 

%% file: 02_hoelder_definitions.tex
\section{Function spaces on embedded manifolds and on the torus}\label{sec:hoel}
In this section, we give an overview of the notation used throughout the paper. 
Let $d,d^\prime \in \N$. We write $[d]\coloneqq\{1,2,...,d\}$. For $\bm x \in \C^d$, we denote by
\begin{equation*}
	\abs{\bm x} \coloneqq \sum_{j=1}^d \abs{x_j} \text{ and } \norm{\bm x}\coloneqq \sqrt{\sum_{j=1}^d \abs{x_j}^2}
\end{equation*}
the $1$-norm and the Euclidean norm, respectively.  For $k \in \NN$, we set 
\begin{equation*}
    B_k^d \coloneqq \{\bm{\beta} \in \NN^d \mid \abs{\bm{\beta}} \le k\}.
\end{equation*}

\begin{defin}[Function spaces in $\R^d$] \label{defin:hoel_R}
	Let $k \in \NN$ and  let $U \subset \R^d$ and $V \subset \C^{d^\prime}$, where $U$ is assumed to be open if $k>0$. We define the differentiability space of order $k$ by
	\begin{equation*}
		\mathcal{C}^k(U,V)\coloneqq \mleft\{f \colon U \to V \;\middle|\;  \begin{alignedat}{1}& \mathrm{D}^{\bm \beta} f \text{ exists and is bounded}\\ &\text{and continuous for all }\bm \beta \in B^d_k \end{alignedat}\mright\}
	\end{equation*}
	with the norm
	\begin{equation*}
		\norm{f}_{\mathcal{C}^k(U,V)} \coloneqq \max_{\bm \beta \in B^d_k} \sup_{\bm x \in U} \big\lVert D^{\bm \beta}f(\bm x)\big\rVert.
	\end{equation*}
	For $0<\alpha<1$, the $(k,\alpha)$-Hölder space is
	\begin{equation*}
		\mathcal{C}^{k,\alpha}(U,V) \coloneqq \mleft\{ f \in \mathcal{C}^k(U,V)\;\middle|\; \abs{f}_{\mathcal{C}^{k,\alpha}(U,V)}\coloneqq \sup_{\substack{\bm x,\bm y \in U,\, \bm x \not=\bm y\\\bm \beta \in B^d_k,\, \abs{\bm \beta}=k}}\frac{\norm{\mathrm{D}^{\bm \beta}f(\bm x)-\mathrm{D}^{\bm \beta}f(\bm y)}}{\norm{\bm x-\bm y}^\alpha}<\infty\mright\}
	\end{equation*}
	with the norm
	\begin{equation*}
		\norm{f}_{\mathcal{C}^{k,\alpha}(U,V)}\coloneqq \norm{f}_{\mathcal{C}^k(U,V)}+\abs{f}_{\mathcal{C}^{k,\alpha}(U,V)}.
	\end{equation*}
	Finally, we set the Lipschitz space of order $k$
	\begin{equation*}
		\mathrm{Lip}^k(U,V)\coloneqq \mleft\{ f \in \mathcal{C}^k(U,V)\;\middle|\; \abs{f}_{\mathrm{Lip}^k(U,V)}\coloneqq \sup_{\substack{\bm x, \bm y \in U,\,\bm x \not=\bm y\\\bm \beta \in B^d_k,\,\abs{\bm \beta}=k}}\frac{\norm{\mathrm{D}^{\bm \beta}f(\bm x)-\mathrm{D}^{\bm \beta}f(\bm y)}}{\norm{\bm x-\bm y}}<\infty\mright\}
	\end{equation*}
	with the norm
	\begin{equation*}
		\norm{f}_{\mathrm{Lip}^k(U,V)}\coloneqq \norm{f}_{\mathcal{C}^k(U,V)}+\abs{f}_{\mathrm{Lip}^k(U,V)}.
	\end{equation*}
	All three spaces, equipped with the given norms, are Banach spaces. We denote the space of smooth functions with bounded partial derivatives by
	\begin{equation*}
		\mathcal{C}^\infty(U,V) \coloneqq \bigcap_{k \in \NN} \mathcal{C}^k(U,V).
	\end{equation*}
\end{defin}

The corresponding function spaces on the torus $\T^d= \R^d /(2 \pi \Z^d)$ are obtained by restricting the function spaces on $\R^d$ to $2\pi$-periodic functions:

\begin{defin}[Function spaces on $\T^d$]
	Let $\mathcal{X}$ be any of the function spaces $\mathcal{C}^k$, $\mathcal{C}^{k,\alpha}$, $\mathrm{Lip}^k$ or $\mathcal{C}^\infty$ from \Cref{defin:hoel_R}, and let $V \subset \C^{d^\prime}$.
  We define the respective function space on the torus $\T^d$ by
  \begin{equation*}
    \mathcal{X}(\T^d,V)\coloneqq \{f \in \mathcal{X}(\R^d,V) \mid f(2\pi \bm{e}^j+\cdot)=f(\cdot) \text{ for all } j\in [d]\},
  \end{equation*}
  where $\bm e^j$ denotes the $j$-th unit vector for $j \in [d]$.
\end{defin}

When considering scalar-valued functions in either of the previous definitions, i.e., when we have $V=\C$, then $V$ is usually omitted.

\begin{defin}[Function spaces on embedded manifolds]\label{defin:manifold}
	Let $\M \subset \R^{d^\prime}$ be a smooth embedded submanifold with or without corners. For $f \colon \M \to \C$, we call $f^\ast \colon U \to \C$ an extension of $f$ if $U$ is an open set in $\R^{d^\prime}$ with $\M \subset U \subset \R^{d^\prime}$ and we have $\restr{f^\ast}{\M}=f$. For $k \in \NN$, we call $f^\ast$ a $\mathcal{C}^k$-extension of $f$, if it is an extension of $f$ and $f^\ast \in \mathcal{C}^k(U)$. Then, the $\mathcal{C}^k$-extension seminorm of $f^\ast$ is
	\begin{equation}\label{eq:C_semi_manifold}
		\abs{f^\ast}_{\mathcal{C}^k(\M)}^\ast \coloneqq \sup_{\bm \beta \in B^{d^\prime}_k} \big\lVert\mathrm{D}^{\bm \beta}f^\ast \big\rVert_{\mathcal{C}(\M)}.
	\end{equation}
	We define the differentiability space of order $k$ on $\M$ by
	\begin{equation}\label{eq:C_norm_manifold}
		\mathcal{C}^k(\M) \coloneqq \mleft\{f \colon \M \to \C \;\middle|\; \norm{f}_{\mathcal{C}^k(\M)}\coloneqq \inf_{\substack{\mathcal{C}^k\text{-extensions}\\ f^\ast \text{ of }f}} \abs{f^\ast}_{\mathcal{C}^k(\M)}^\ast<\infty\mright\}
	\end{equation}
	and write $\mathcal{C}^\infty(\M)\coloneqq \bigcap_{k \in \NN} \mathcal{C}^k(\M)$.
	
	Analogously, for $0<\alpha<1$, we call $f^\ast$ a $\mathcal{C}^{k,\alpha}$-extension of $f$, if it is an extension of $f$ and $ f^\ast \in \mathcal{C}^{k,\alpha}(U)$. The $\mathcal{C}^{k,\alpha}$-extension seminorm is then given by
	\begin{equation}\label{eq:H_semi_manifold}
		\abs{f^\ast}_{\mathcal{C}^{k,\alpha}(\M)}^\ast \coloneqq \abs{f^\ast}_{\mathcal{C}^k(\M)}^\ast+\sup_{\substack{\bm \xi,\bm \eta \in \M, \, \bm \xi \not=\bm \eta\\\bm \beta \in B^{d^\prime}_k, \,\abs{\bm \beta}=k}}\frac{\norm{\mathrm{D}^{\bm \beta}f^\ast (\bm \xi)-\mathrm{D}^{\bm \beta}f^\ast(\bm \eta)}}{\norm{\bm \xi-\bm \eta}^\alpha}.
	\end{equation}
	Finally, we define the $(k,\alpha)$-Hölder space on $\M$ by
	\begin{equation}\label{eq:H_norm_manifold}
		\mathcal{C}^{k,\alpha}(\M) \coloneqq \mleft\{f \colon \M \to \C \;\middle|\; \norm{f}_{\mathcal{C}^{k,\alpha}(\M)}\coloneqq \inf_{\substack{\mathcal{C}^{k,\alpha}\text{-extensions}\\ f^\ast \text{ of }f}} \abs{f^\ast}_{\mathcal{C}^{k,\alpha}(\M)}^\ast<\infty\mright\}.
	\end{equation}
\end{defin}

%% file: 03_DFS_method.tex
\section{The generalized DFS method}\label{sec:dfs}

We define a generalization of the classical DFS method to other manifolds in a way that covers generalizations from the literature, such as ball and cylinder, and yields results analogous to those we presented in \cite{MilQue22} for the sphere. 

The classical double Fourier sphere (DFS) method transforms a function defined on the sphere 
$
	\S^2 \coloneqq \{\bm x \in \R^3 \mid \norm{\bm x}=1\}
$
to the torus $\T^2$ and subsequently represents it via a Fourier series. Thereby, a function $f \colon \S^2 \to \C$ is concatenated with the {DFS coordinate transform}
\begin{equation*}
	\phi_{\S^2} \colon \T^2 \to \S^2, \, (x_1,x_2) \mapsto (\cos x_1 \sin x_2,\sin x_1\sin x_2,\cos x_2),
\end{equation*}
which covers the sphere twice. The transform $\phi_{\S^2}$ is smooth, and the transformed function $f \circ \phi_{\S^2}$ has a convergent Fourier series for sufficiently smooth $f$, see \cite{MilQue22}. Furthermore, we have $\phi_{\S^2}(x_1,x_2)=\phi_{\S^2}(x_1+\pi,-x_2)$ for $(x_1,x_2) \in \T^2$, so that the transformed function is block-mirror-centro\-symmetric (BMC), cf.\ \cite[§~2.2]{TowWilWri16}, as illustrated in \Cref{fig:DFS}.
\begin{figure}[ht]
	\hspace{1.5cm}
  \includegraphics[height=12em]{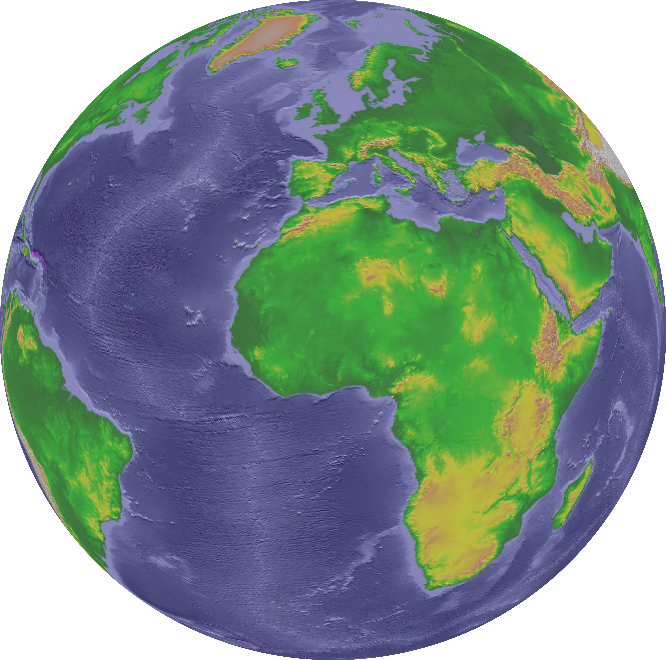}
  \hspace{1.25cm}
  \includegraphics[height=12em]{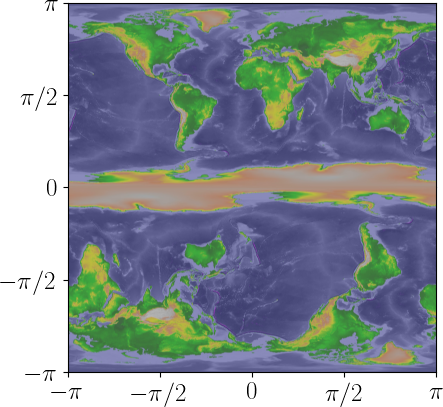}
  \caption{Left: Topographic data $f(\bm\xi)$ of the earth. Right: DFS function $f\circ\phi_{\S^2}(\bm x)$.}
    \label{fig:DFS}
\end{figure}

Because the restriction of $\phi_{\S^2}$ to $(-\pi,\pi]\times (0,\pi) \cup \{(0,0),(0,\pi)\}$ is bijective, there is a one-to-one connection between BMC-functions on the torus and functions on the sphere without its poles. This makes it possible to relate the Fourier expansion of a transformed function $f \circ \phi_{\S^2}$ to a series expansion of $f$ defined directly on the sphere, see \cite{MilQue22}.

The core concept of our generalization of this method is to transform a function defined on a $d$-dimensional manifold to a function on the torus $\T^d$. The transformed function can then be represented via a Fourier series, which allows for fast numerical computations by the FFT on the torus. To ensure similar properties as in the classical case, we impose smoothness and symmetry assumptions on the transform.

\begin{defin}\label{defin:DFS}
    Let $d,{d^\prime} \in \N$ and let $\M \subset \R^{d^\prime}$ be a $d$-dimensional smooth embedded submanifold with or without corners. We call a surjective function     
    \begin{equation*}
        \phi \colon \T^d \to \M 
    \end{equation*}
    a \emph{generalized DFS transform} of $\M$ if it fulfills the following smoothness and symmetry assumptions:
    We say $\phi$ has the \emph{smoothness properties} of a DFS transform if $\phi \in \mathcal{C}^\infty(\T^d)$ and for all $\bm{\mu}\in \NN^d$ and $l \in [{d^\prime}]$, it holds that
    \begin{equation}\label{eq:DFSinf}
        \norm{\mathrm D^{\bm{\mu}} \phi_l}_{\mathcal{C}(\T^d)} \le 1.
    \end{equation}
    
    We say that $\phi$ has the \emph{symmetry properties} of a DFS transform if, firstly, for some integer $p \in \NN$, shift vectors $\mathrm S^i \in \{0,\pi\}^d, \, i \in [p]$, and reflection maps
    \begin{equation}\label{eq:Mi} 
    	\mathrm M^i \colon \T^d \to \T^d, \, \mathrm M^i(\bm x) \coloneqq \mathrm M^i \bm x, \qquad i \in [p],
	\end{equation}
	associated to some diagonal matrices $\mathrm M^i\in\Z^{d\times d}$ with diagonal entries in $\{-1,1\}$, the map $\phi$ is invariant under the \emph{symmetry functions} 
	\begin{equation}\label{eq:si}
		s^i \colon \T^d \to \T^d, \, s^i(\bm x) \coloneqq \mathrm S^i+ \mathrm M^i(\bm x), \qquad i \in [p],	\end{equation}
		i.e., $\phi\circ s^I = \phi$ for all $I\subset[p]$, where the repeated composition of functions is written as 
    $$
    s^{\{i_1, \, i_2,\,...\}}\coloneqq
    \circ_{i \in  \{i_1,\,i_2,\,...\} } s^{i} \coloneqq s^{i_1}\circ \pare{s^{i_2} \circ ...}, \qquad  \{i_1,\, i_2,\,...\} \subset [p],
    $$ 
    with the convention that the empty composition is the identity.
	Secondly, for the symmetry properties to be satisfied, there must exist a rectangular set $D \subset \T^d$ of representatives of $\T^d/_{\sim}$, where $\sim$ is the equivalence relation $\bm x \sim \bm y \iff \bm y \in \{s^I(\bm x) \mid I \subset  [p]\}$, and disjoint measurable subsets $D_1, D_2 \subset D$ such that
	\begin{enumerate}[label=(\roman*)]
		\item $D^\circ=D_1^\circ$ and $\phi[D_2]$ is closed, \label{Prop1}
		\item the restriction $\restr{\phi}{D_1 \dot{\cup} D_2} \colon D_1 \dot{\cup} D_2 \to \M$ is a bijection on the disjoint union $D_1 \dot{\cup} D_2$,\label{Prop2}
		\item the inverse $\big(\restr{\phi}{D_1}\big)^{-1}\colon \M \setminus \phi[D_2] \to D_1$ is continuous, and \label{Prop3}
		\item the Jacobian $\nabla \phi(\bm x)$ has full rank for all $\bm x \in D_1^\circ$.  \label{Prop4}
    \end{enumerate}
    The set $D$ being ``rectangular'' is to be understood as it being the Cartesian product of connected subsets of $\T^1$, i.e., it can be identified with a rectangle in $\R^d$.
    We always assume $p$ to be chosen minimally and call it the \emph{symmetry number} of $\phi$.
    
    For $f \colon \M \to \C$, we define the \emph{generalized DFS function} (with respect to $\phi$) of $f$ by
    \begin{equation*}
    	\tilde f \colon \T^d \to \M, \, \tilde f \coloneqq f \circ \phi.
    \end{equation*}
\end{defin}

The matrix--vector product \eqref{eq:Mi} of $\mathrm M^i$ and $\bm x \in \T^d$ can be performed with any representative of $\bm x$ in $\R^d$ since the matrix $\mathrm M^i$ only has integer entries.
The addition of shifts $\mathrm{S}^i$ and the reflections $\mathrm{M}^i$ are, as functions on the torus, self-inverse and commute, so we do not need to consider the order of compositions. We have $s^I(\cdot)=\sum_{i \in I}\mathrm S^i+\mathrm M^I(\cdot)$, and $\sim$ as defined above is indeed an equivalence relation on $\T^d$. Because $D$ is a set of representatives of $\T^d/_\sim$, properties \ref{Prop1} and \ref{Prop2} together with the invariance assumption on $\phi$  imply that the symmetry functions $s^i$, $i \in [p]$, are unique up to compositions.

\begin{rem}
We consider the class of smooth submanifolds with corners because the finite Cartesian product of smooth manifolds with corners is again a smooth manifold with corners, whereas the same is not true for smooth manifolds with boundary, as their product might lack a smooth structure in the right sense, cf.\ \cite[p.~29]{Lee12}. Thus, choosing this class of manifolds allows us to generate DFS methods on product manifolds, such as the cylinder, in \Cref{sec:product}. As the set of corner points or the boundary of such manifold might be empty, cf.\ \cite[p.~26 \& p.~417]{Lee12}, we usually write ``with or without corners''.
\end{rem}

\begin{rem}
The bound in \eqref{eq:DFSinf} is somewhat arbitrary and arises from the specific applications. We could instead allow for any uniform bound on the partial derivatives, i.e., consider $\tilde{\phi} \in \mathcal{C}^\infty(\T^d)$ with $\tilde{\phi}[\T^d]=\M$ and some $C>0$ such that
	$
		\norm{\smash{\mathrm D^{\bm{\mu}} \tilde\phi_l}}_{\mathcal{C}(\T^d)} \le C
  $ 
  for all $\bm{\mu} \in \NN^d$ and $l \in [{d^\prime}]$. 
  The results in this paper can then be applied to a function $f \colon \M \to \C$ by transforming it to the scaled manifold $C^{-1} \M$.
\end{rem}

The next lemma provides some basic properties of generalized DFS transforms: As $\phi$ is surjective, a right inverse always exists and can be chosen canonically by \ref{Prop2} in \Cref{defin:DFS}. This inverse has certain regularity properties due to \ref{Prop3} and \ref{Prop4}.

\begin{lem}\label{lem:even}
	Let $\M \subset \R^{d^\prime}$ be a smooth embedded submanifold with or without corners that has generalized DFS transform $\phi \colon \T^d \to \M$. Let $p$, $D$, $D_1$, $D_2$, and $s^i$, $i \in [p]$, be as in \Cref{defin:DFS}. Then $\phi$ evenly covers $\phi[D_1^\circ]$ in the sense that 
	\begin{equation}\label{eq:even}
		s^I(D_1^\circ) \cap s^J(D_1^\circ)=\emptyset,\qquad I, \, J \subset [p] \text{ with } I\not=J.
	\end{equation} 
	Furthermore, $\phi[D_2]$ and $\M \setminus \phi[D_1^\circ]$ have measure zero in $\M$ and $\T^d \setminus \big(\bigcup_{I \subset [p]} s^I(D_1^\circ)\big)$ has measure zero in $\T^d$. 
	 The inverse of $\restr{\phi}{D_1 \dot{\cup} D_2}$ is continuous on $\M \setminus \phi[D_2]$ and smooth in the manifold interior of $\M$ without $\phi[D_2]$, i.e., for any $\bm \xi \in \M \setminus \phi[D_2]$ in the manifold interior of $\M$, there exists a neighborhood $U$ of $\bm \xi$ in $\M$ such that $\restr{\smash{(\restr{\phi}{D_1\dot{\cup} D_2})^{-1}}}{U} \in \mathcal{C}^\infty(U)$. We call $\phi[D_2]$ the \emph{set of singularities}.
\end{lem}

\begin{proof}
	For \eqref{eq:even}, one can show that $s^I(\bm x)\not=\bm x$ for all $\bm x \in D^\circ$ and non-empty $I \subset [p]$.
	Furthermore, for $\bm x, \bm y \in D$ with $s^I(\bm x)=\bm y$ for some $I \subset [p]$, we must have $\bm x=\bm y$, as $D$ is a set of representatives of $\T^d/_\sim$, with $\sim$ as in \Cref{defin:DFS}. Thus, we have $s^I(D^\circ)\cap D^\circ=\emptyset$ for all $I \subset [p]$. Clearly, for any $I \subset [p]$ the set $s^I(D)$ is also a set of representatives of $\T^d/_\sim$ and, since $s^I$ is a diffeomorphism, $s^I(D^\circ)$ is open. Thus, we can use the same arguments to show that $s^I(D^\circ)\cap s^J(D^\circ)=\emptyset$ for all $I, \, J \subset [p]$ with $I\not=J$, this is \eqref{eq:even}.
	
	The boundary of a rectangular subset of $\T^d$ is a set of measure zero in $\T^d$. Thus, \ref{Prop1} in \Cref{defin:DFS} implies that $D_2 \subset D \setminus D_1^\circ \subset \partial D$ is a set of measure zero in $\T^d$. In particular, $\phi[D_2]$ and $\M \setminus \phi[D_1^\circ]\subset \phi[\partial D]$ have measure zero in $\M$ since $\phi$ is smooth, cf.\  \cite[thm.~6.9]{Lee12}. Furthermore, we have for all $I,J \subset [p]$ with $I \not=J$ that $s^I(D_1^\circ)\subset s^I(D)$ and $s^J(D_1^\circ)\cap s^I(D)=\emptyset$, where we used \eqref{eq:even} and the fact that $s^I(D)$ and $s^J(D)$ are rectangular. Since $D$ is a set of representatives of $\T^d/_\sim$, we obtain that
	\begin{equation*}
		\T^d \setminus \big( \bigcup_{I \subset [p]} s^I(D_1^\circ)\big)=\bigcup_{I \subset [p]} \big(s^I(D)\setminus  s^I(D_1^\circ)\big) \subset \bigcup_{I \subset [p]} s^I(\partial D),
	\end{equation*}
	are sets of measure zero in $\T^d$.
 
	The continuity of $\big(\restr{\phi}{D_1\dot{\cup} D_2}\big)^{-1}$ in $\M \setminus \phi[D_2]$ follows immediately form \ref{Prop2} and \ref{Prop3} in \Cref{defin:DFS}. For the smoothness in the manifold interior, we observe that, by \ref{Prop1}, the set $\M \setminus \phi[D_2]$ is open in $\M$ and thus a smooth submanifold with boundary of $\R^{d^\prime}$. Since homeomorphisms preserve manifold boundaries, \ref{Prop3} implies that $\phi[D_1^ \circ]$ is the manifold interior of $\M \setminus \phi[D_2]$, in particular $\phi[D_1^\circ]$ is a smooth submanifold without boundary of $\R^{d^\prime}$. Similarly $D_1^\circ$ is open in $\T^d$ and thus a smooth manifold without boundary. By \ref{Prop4}, $\nabla \phi(\bm x)$ has full rank for all $\bm x \in D_1^\circ$, thus we can apply the inverse function theorem \cite[thm.~4.5]{Lee12} to $\restr{\phi}{D_1^\circ}$. We obtain that $\restr{\phi}{D_1^\circ}$ is a local diffeomorphism, in particular the inverse of $\restr{\phi}{D_1^\circ}$ is smooth, in the sense that coordinate representations are infinitely differentiable, and can thus be extended to a function defined on an open neighborhood of $\phi[D_1^\circ] \subset \R^{d^\prime}$ that has continuous partial derivatives of all orders. However, these derivatives might be unbounded, thus we obtain $\big(\restr{\phi}{D_1^\circ}\big)^{-1} \in \mathcal{C}^\infty$ only locally, cf. \Cref{defin:hoel_R,defin:manifold}.
\end{proof}

As in the special case of the classical DFS method, the symmetry properties of the generalized method impose symmetry upon the DFS functions. Furthermore, functions on the torus that are invariant under the DFS symmetry functions correspond to functions well-defined on the manifold without the set of singularities. This relationship is formalized in the following lemma, which will later allow us to transfer the series approximation of transformed functions on the torus back to the manifold.

\begin{lem}\label{lem:BMC}
	Let $\phi \colon \T^d \to \M$ be a generalized DFS transform with symmetry number $p$ and let $D_1$, $D_2$, and $s^i$, $i \in [p]$, be as in \Cref{defin:DFS}. We call some function $g \colon \T^d \to \C$ a \emph{BMC function (of type $\phi$)} if it is invariant under the symmetry functions $s^i$, i.e., $g = g\circ s^i$ for all $i \in [p]$. 
  For any $f\colon \M\to\C$, its DFS function $\tilde f$ is a BMC function.
  Conversely, if $g \colon \T^d \to \C$ is a BMC function, then there exists $f \colon \M \to \C$ such that
	\begin{equation*}
		\tilde{f}(\bm x)=g(\bm x), \qquad \bm x \in \bigcup_{I \subset [p]} s^I[D_1].
	\end{equation*}
	All possible choices of such $f$ coincide on $\M \setminus \phi[D_2]$. Setting $f \coloneqq g \circ ({\restr{\phi}{D_1\dot{\cup} D_2}})^{-1}$ yields the unique $f$ that also satisfies $\tilde{f}(\bm x)=g(\bm x)$ for $\bm x \in D_2$.
\end{lem}

\begin{proof}
	By \Cref{defin:DFS}, we know that $\phi$ is $s^i$-invariant for any $i \in [p]$. Thus, any DFS function is a BMC function. By \ref{Prop2} in \Cref{defin:DFS}, the transform $\phi$ bijectively maps $D_1 \dot{\cup} D_2$ to $\M$. In particular, the function $f \coloneqq g \circ (\restr{\phi}{D_1\dot{\cup} D_2})^{-1}$ is well-defined and it is clearly the unique choice of function whose DFS function coincides with $g$ on $D_1 \dot{\cup} D_2$. If $g$ is a BMC function both $\tilde{f}$ and $g$ are invariant under the symmetry functions and thus the equality extends to $\cup_{I \subset [p]} s^I[D_1\dot{\cup} D_2]$. Inversely, if $\tilde{f_1}(\bm x)=g(\bm x)=\tilde{f_2}(\bm x)$ for $\bm x \in D_1$, then we have for $\bm \xi  \in \phi[D_1]=\M \setminus \phi[D_2]$ that
	\begin{equation*}
		f_1(\bm \xi)=(f_1\circ \phi)\big((\restr{\phi}{D_1})^{-1}(\bm \xi)\big)=(f_2\circ \phi)\big((\restr{\phi}{D_1})^{-1}(\bm \xi)\big)=f_2(\bm \xi).\qedhere
	\end{equation*}
\end{proof}

We need the following proposition, whose proof can be found in \cite[p.~7]{MilQue22}.

\begin{prop}\label{prop:lipH}
	Let $U \subset \R^d$ be an open set, $V \subset U$, and $g\colon U \to \C$. If $g$ is bounded and Lipschitz-continuous on $V$, then $g$ is $\alpha$-Hölder continuous on $V$ for all $0<\alpha<1$ with 
	\begin{equation}\label{eq:lipH1}
		\abs{g}_{\mathcal{C}^{\alpha}(V)} \le \max \left\{\abs{g}_{\mathrm{Lip}(V)}, \, 2\norm{g}_{\mathcal{C}(V)} \right\}.
	\end{equation} 
	Furthermore, if $V$ is convex and $g\in\mathcal{C}^1(U)$, then $g$ is Lipschitz-continuous on $V$ with
	\begin{equation}\label{eq:lipH2}
		\abs{g}_{\mathrm{Lip}(V)} \le \norm{\nabla g}_{\mathcal{C}(U,\C^d)}.
	\end{equation}
\end{prop}

Utilizing the smoothness \eqref{eq:DFSinf} of the generalized DFS transform, we can immediately conclude the following properties.

\begin{cor}\label{lem:DFS}
	Let $\phi \colon \T^d \to \M$ be a generalized DFS transform of some smooth embedded submanifold with or without corners $\M$ of $\R^{d^\prime}$. Then, for all $\bm{\mu} \in \NN^d$ and $l \in [{d^\prime}]$ we have $\mathrm D^{\bm{\mu}}\phi_l \in \mathrm{Lip}(\T^d)$ and $\mathrm D^{\bm{\mu}}\phi_l \in \mathcal{C}^\alpha(\T^d)$ with 
	\begin{align}
        & \abs{\mathrm D^{\bm{\mu}} \phi_l}_{\mathrm{Lip}(\T^d)} \le \sqrt{d}, \label{eq:DFSL}\\
		& \abs{\mathrm D^{\bm{\mu}} \phi_l}_{\mathcal{C}^\alpha(\T^d)} \le 2\sqrt{d}.	\label{eq:DFSh}
	\end{align}
\end{cor}

%% file: 04_smoothness.tex
\section{Hölder continuity of DFS functions}\label{sec:smoothness}

In this section, we show that the DFS transform preserves Hölder-smoothness. More precisely, it maps the function spaces $\mathcal{C}^{k+1}(\M)$ and $\mathcal{C}^{k,\alpha}(\M)$ into the Hölder space $\mathcal{C}^{k,\alpha}(\T^d)$. We prove respective norm bounds. 
In \Cref{sec:series}, we will utilize these findings to obtain convergence rates of the series representation of the DFS function. 
The results in this section only require the smoothness \eqref{eq:DFSinf} of the DFS transform and are straightforward generalizations of the work \cite[§~4]{MilQue22} on the classical DFS method.

\input{04.1_faaDi}
\input{04.2_hoelderDiff}
\input{04.3_diffDiff}

\begin{rem}
	For the special case of the sphere $\M = \S^2$, \cite[thm.\ 4.3]{MilQue22} states that
	$
		 \abs{\smash{\tilde f}}_{\mathcal{C}^{k,\alpha}(\T^2)} \le (k+3)! \norm{f}_{\mathcal{C}^{k+1}(\S^2)}
  $
  for 
  $
  f \in \mathcal{C}^{k+1}(\S^2).
	$
	The corresponding result \eqref{eq:hC_bound} in this paper improves the estimate by the factor $\sqrt{2}$. On the other hand \cite[thm.\ 4.5]{MilQue22} states that
	$
		\abs{\smash{\tilde f}}_{\mathcal{C}^{k,\alpha}(\T^2)} \le (k+3)! \norm{f}_{\mathcal{C}^{k,\alpha}(\S^2)}
  $ 
  for
  $ 
    f\in \mathcal{C}^{k,\alpha}(\S^2).
	$
	Comparing this with \eqref{eq:hH_bound} for the sphere, we observe that the new general estimate is larger by a factor $\sqrt{2}$. This is due to \eqref{eq:DFSh} not being optimal for the spherical DFS transform,  in fact \cite[lem.\ 4.2]{MilQue22} proves that the respective estimate in this special case.
\end{rem}

%% file: 04.1_faaDi.tex
The following technical lemma, which is proven in \Cref{app:proof}, bounds the number of summands in the multivariate chain rule for higher partial derivatives of vector-valued functions.

\begin{lem}\label{lem:faaDi}
	For $d, \,{d^\prime}\in\N$ and $k \in \NN$, let $h \colon U \to V$ and $g \colon V \to \mathbb{C}$ be $k$-times continuously differentiable functions defined on some open sets $U \subset \mathbb{R}^d$ and $V \subset \mathbb{R}^{d^\prime}$, respectively. Then, for any $\bm{\beta} \in B^d_k$, we have
    \begin{equation}\label{eq:indH1}
        \mathrm D^{\bm{\beta}} (g \circ h)=  \sum_{i=1}^n \pare{\mathrm D^{\bm{\gamma}_i}g \circ h} \, \prod_{j=1}^{m_i} \mathrm D^{\bm{\mu}_{ij}}h_{\ell_{ij}}
    \end{equation}
    for some constants depending on $\bm{\beta}$, which fulfill
    \begin{align}
        n &\in \NN, \, n \leq \tfrac{(k+{d^\prime}-1)!}{({d^\prime}-1)!},\label{eq:indH2} \\
        m_i &\in \NN, \, m_i \le k, \, \, i \in [n],\label{eq:indH4}\\
        \bm{\gamma}_i &\in B^{d^\prime}_k, \, i \in [n],\label{eq:indH3} \\
        \bm{\mu}_{ij} &\in B^d_k, \, i\in [n], \, j \in [m_i],\label{eq:indH5}\\
        \ell_{ij} &\in [{d^\prime}], \, i \in [n], \, j \in [m_i].\label{eq:indH6}
    \end{align}
\end{lem}

%% file: 04.2_hoelderDiff.tex
\begin{thm}\label{thm:hC_bound}
    Let $\M\subset\R^{d^\prime}$ be a smooth embedded submanifold with or without corners
that admits a generalized DFS transform $\phi \colon \T^d \to \M$. For $k \in \NN$ and $f \in \mathcal{C}^{k+1}(\M)$, the generalized DFS function $\tilde{f}=f \circ \phi$ is in $\mathcal{C}^{k,\alpha}(\T^d)$ for all $0<\alpha<1$. If $k+d^\prime\ge 2$, we have
    \begin{equation}\label{eq:hC_bound}
        \big\lvert \tilde{f}\big\rvert_{\mathcal{C}^{k,\alpha}(\T^d)}\le \sqrt{d} \, \frac{(k+{d^\prime})!}{({d^\prime}-1)!} \, \norm{f}_{\mathcal{C}^{k+1}(\M)},
    \end{equation}
    and if $k+d^\prime=1$, then $\big\lvert \tilde{f}\big\rvert_{\mathcal{C}^\alpha(\T^1)} \le 2 \norm{f}_{\mathcal{C}^1(\M)}$.
    Furthermore, it holds that $\tilde{f} \in \mathcal{C}^{k+1}(\T^d)$ with
    \begin{equation} \label{eq:Ck_norm}
     \big\lVert \tilde f \big\rVert_{\mathcal{C}^{k+1}(\T^d)}\le \frac{(k+{d^\prime})!}{({d^\prime}-1)!}\norm{f}_{\mathcal{C}^{k+1}(\M)}.
    \end{equation}
\end{thm}

\begin{proof}
	We first prove 
	\begin{equation}\label{eq:C_bound}
		\big\lVert \mathrm D^{\bm{\beta}}\tilde{f} \big\rVert_{\mathcal{C}(\R^d)}\le \frac{(k^\prime+{d^\prime}-1)!}{({d^\prime}-1)!} \, \norm{f}_{\mathcal{C}^{k^\prime}(\M)}
	\end{equation}
	for all $k^\prime \in \NN$ with $k^\prime \le k+1$ and all $\bm{\beta} \in B_{k^\prime}^d$, which then also implies \eqref{eq:Ck_norm}. Let $f^\ast \in \mathcal{C}^{k^\prime}(U)$ be a $\mathcal{C}^{k^\prime}$-extension of $f$, where $U \subset \M$ is some open set as in \Cref{defin:manifold}. Since $\phi$ satisfies $\phi[\mathbb{R}^d] = \M \subset U$, it can be considered as a function $\phi\colon\R^d\to U$. Thus, we can apply \Cref{lem:faaDi} to $\tilde{f}=f^\ast \circ \phi$ and obtain $\tilde{f} \in \mathcal{C}^{k^\prime}(\R^d)$ and 
    \begin{equation*}
        \mathrm D^{\bm{\beta}} \tilde{f} = \mathrm D^{\bm{\beta}} (f^* \circ \phi)= \sum_{i=1}^n \pare{\mathrm D^{\bm{\gamma}_i}f^* \circ \phi} \, \prod_{j=1}^{m_i} \mathrm D^{\bm{\mu}_{ij}}\phi_{\ell_{ij}}
    \end{equation*}
    for some constants satisfying \eqref{eq:indH2} to \eqref{eq:indH6} for $k^\prime$. This implies that for all $\bm{x} \in \R^d$
    \begin{alignat*}{3}
        &\big\lvert \mathrm D^{\bm{\beta}}\tilde{f}(\bm{x}) \big\rvert
        && \centermathcell{\le}  &&\sum_{i=1}^n \big\lvert \pare{\mathrm D^{\bm{\gamma}_i}f^* \circ \phi}(\bm{x})\big\rvert \, \prod_{j=1}^{m_i} \abs{\mathrm D^{\bm{\mu}_{ij}}\phi_{\ell_{ij}}(\bm{x})}\\
    		& && \centermathcell{\underset{\eqref{eq:C_semi_manifold},\eqref{eq:DFSinf}}{\le}} && n \, \abs{f^*}_{\mathcal{C}^{k^\prime}(\M)}^* \underset{\eqref{eq:indH2}}{\le}\frac{(k^\prime+{d^\prime}-1)!}{({d^\prime}-1)!} \, \abs{f^*}_{\mathcal{C}^{k^\prime}(\M)}^*.
    \end{alignat*}
	Since this bound holds for any $\mathcal{C}^{k^\prime}$-extension $f^*$ of $f$,  we can replace $\abs{f^*}_{\mathcal{C}^{k^\prime}(\M)}^*$ by $\norm{f}_{\mathcal{C}^{k^\prime}(\M)}$, see \eqref{eq:C_norm_manifold}, on the right hand side. This proves \eqref{eq:C_bound}. Next, we show
	\begin{equation}\label{eq:L_bound}
		\big\lvert \mathrm D^{\bm{\beta}}\tilde{f}\big\rvert_{\mathrm{Lip}(\R^d)}\le \sqrt{d} \, \frac{(k+{d^\prime})!}{({d^\prime}-1)!} \, \norm{f}_{\mathcal{C}^{k+1}(\M)} 
	\end{equation}
	for all $\bm{\beta}\in B_k^d$. We know that $\mathrm D^{\bm{\beta}}\tilde{f}$ is continuously differentiable and by \eqref{eq:C_bound} for $k^\prime=k+1$, we obtain
    \begin{align*}
        \big\lVert \nabla \big(\mathrm D^{\bm{\beta}}\tilde{f} \,\big)\big\rVert_{\mathcal{C}(\R^d,\R^d)} & = \sup_{\bm{x} \in \R^d} \sqrt{\sum_{p=1}^d \big\lvert \big( \mathrm D^{\bm{e}^p+\bm{\beta}}\tilde{f} \, \big)(\bm{x})\big\rvert^2}\\
        & \le \sqrt{d \pare{\frac{(k+{d^\prime})!}{({d^\prime}-1)!} \, \norm{f}_{\mathcal{C}^{k+1}(\M)}}^2}
        = \sqrt{d} \, \frac{(k+{d^\prime})!}{({d^\prime}-1)!} \,  \norm{f}_{\mathcal{C}^{k+1}(\M)}.
    \end{align*}
    Together with \eqref{eq:lipH2} this proves \eqref{eq:L_bound}.
    Combining \eqref{eq:C_bound} for $k^\prime=k$ with \eqref{eq:L_bound} and applying \eqref{eq:lipH1}, we conclude for all $\bm{\beta}\in B^d_k$ that $\mathrm D^{\bm{\beta}}\tilde{f}$ is $\alpha$-Hölder continuous. We obtain
    \begin{align*}
        \big\lvert \mathrm D^{\bm{\beta}}\tilde{f}\big\rvert_{\mathcal{C}^\alpha(\R^d)} & \le \max \left\{ \big\lvert\mathrm D^{\bm{\beta}}\tilde{f}\big\rvert_{\mathrm{Lip}(\R^d)},2 \, \big\lVert \mathrm D^{\bm{\beta}}\tilde{f}\big\rVert_{\mathcal{C}(\R^d)}\right\}\\
        & \le \max \left\{ \sqrt{d} \, \frac{(k+{d^\prime})!}{({d^\prime}-1)!} \,  \norm{f}_{\mathcal{C}^{k+1}(\M)},2 \, \frac{(k+{d^\prime}-1)!}{({d^\prime}-1)!} \,  \norm{f}_{\mathcal{C}^{k}(\M)} \right\}\\
        & \le \max \left\{ \sqrt{d}, \frac{2}{k+{d^\prime}}\right\} \, \frac{(k+{d^\prime})!}{({d^\prime}-1)!} \, \norm{f}_{\mathcal{C}^{k+1}(\M)},
    \end{align*}
    where we used that $\norm{f}_{\mathcal{C}^k(\M)} \le \norm{f}_{\mathcal{C}^{k+1}(\M)}$. Since the right hand side is independent of $\bm{\beta}$, it follows that $\tilde{f}$ is $(k,\alpha)$-Hölder continuous with
    \begin{equation}\label{eq:exact_bound}
        \big\lvert \tilde{f} \big\rvert_{\mathcal{C}^{k,\alpha}(\T^d)}  =\sup_{\bm{\beta}\in B^d_k, \abs{\bm{\beta}}=k} \big\lvert \mathrm D^{\bm{\beta}}\tilde{f}\big\rvert_{\mathcal{C}^\alpha(\R^d)} \le \max \left\{ \sqrt{d}, \frac{2}{k+{d^\prime}}\right\} \, \frac{(k+{d^\prime})!}{({d^\prime}-1)!} \, \norm{f}_{\mathcal{C}^{k+1}(\M)}.
    \end{equation}
    We note that $\sqrt{d}\ge {2}/({k+d^\prime})$ holds if $k+d^\prime\geq 2$ and otherwise we have $d=d^\prime=1$ and $k=0$. Thus, \eqref{eq:exact_bound} proves the theorem.
\end{proof}

%% file: 04.3_diffDiff.tex
\begin{thm}\label{thm:hH_bound}
	Let $\M \subset \R^{d^\prime}$ be a smooth embedded submanifold with or without corners that admits a generalized $\phi \colon \T^d \to \M$. For $k \in \N$, $0<\alpha<1$, and $f \in \mathcal{C}^{k,\alpha}(\M)$, the generalized DFS function $\tilde{f}=f \circ \phi$ is in $\mathcal{C}^{k,\alpha}(\T^d)$. If $k+d^\prime\ge2$, we have
	\begin{equation}\label{eq:hH_bound}
		\big\lvert \tilde{f}\big\rvert_{\mathcal{C}^{k,\alpha}(\T^d)} \le 2 \sqrt{d} \, \frac{(k+{d^\prime})!}{({d^\prime}-1)!} \, \norm{f}_{\mathcal{C}^{k,\alpha}(\M)},
	\end{equation}
	and if $k+d^\prime=1$, then $\big\lvert \tilde f\big\rvert_{\mathcal{C}^\alpha(\T^1)} \le \norm{f}_{\mathcal{C}^\alpha(\M)}$.
\end{thm}

\begin{proof}
	Let $f^\ast \in \mathcal{C}^{k,\alpha}(U)$ be a $\mathcal{C}^{k,\alpha}$-extension of $f$. We first prove 
	\begin{equation}\label{eq:H_bound}
		\abs{\mathrm D^{\bm{\gamma}}f^\ast \circ \phi}_{\mathcal{C}^\alpha(\T^d)}\le \max \left\{ \sqrt{d} \, {d^\prime}, 2\right\} \, \abs{f^\ast}^\ast_{\mathcal{C}^{k,\alpha}(\M)}
	\end{equation}
	for all $\bm{\gamma} \in B^{d^\prime}_k$. For $\abs{\bm{\gamma}}=k$ this follows from the definition of the extension seminorm since we have for all $\bm{x},\bm{y} \in \R^d$ that
		\begin{alignat*}{3}
			& \abs{\pare{\mathrm D^{\bm{\gamma}}f^\ast \circ \phi}(\bm{x})-\pare{\mathrm D^{\bm{\gamma}}f^\ast \circ \phi}(\bm{y})} && \underset{\eqref{eq:H_semi_manifold}}{\le} && \abs{f^\ast}^\ast_{\mathcal{C}^{k,\alpha}(\M)} \, \norm{\phi(\bm{x})-\phi(\bm{y})}^\alpha\\
			& && \le && \abs{f^\ast}^\ast_{\mathcal{C}^{k,\alpha}(\M)} \, \Big(\sum_{\ell=1}^{d^\prime} \abs{\phi_\ell(\bm{x})-\phi_\ell(\bm{y})}^2\Big)^{\frac{\alpha}{2}}\\
			& && \centermathcell{\underset{\eqref{eq:DFSL}}{\le}} && \abs{f^\ast}^\ast_{\mathcal{C}^{k,\alpha}(\M)} \, \pare{ d \, {d^\prime} \, \norm{\bm{x}-\bm{y}}^2}^{\frac{\alpha}{2}}\\
			& && \le && \abs{f^\ast}^\ast_{\mathcal{C}^{k,\alpha}(\M)} \, \sqrt{d} \, {d^\prime} \, \norm{\bm{x}-\bm{y}}^\alpha.
		\end{alignat*}
		If $d=d^\prime=1$ and $k=0$, this shows the claim $\abs{\smash{\tilde f}}_{\mathcal{C}^\alpha(\T^1)} \le \norm{f}_{\mathcal{C}^\alpha(\M)}$ in the case $k+d^\prime=1$.
		For $\abs{\bm{\gamma}}<k$, we have $\mathrm D^{\bm{\gamma}}f^\ast \in \mathcal{C}^1(U)$ and thus $\mathrm D^{\bm{\gamma}}f^\ast$ is a $\mathcal{C}^1$ extension of the restriction $\restr{\pare{\mathrm D^{\bm{\gamma}}f^\ast}}{\M}$ and
	\begin{align*}
		\norm{\restr{\pare{\mathrm D^{\bm{\gamma}}f^\ast}}{\M}}_{\mathcal{C}^1(\M)} & \le \abs{\mathrm D^{\bm{\gamma}}f^\ast}^\ast_{\mathcal{C}^1(\M)}=\max_{\ell \in [{d^\prime}]} \big\lVert \mathrm D^{\bm{e}^\ell}\pare{\mathrm D^{\bm{\gamma}}f^\ast}\big\rVert_{\mathcal{C}(\M)}\\
		& \le \max_{\tilde{\bm{\gamma}}\in B^{d^\prime}_k} \big\lVert \mathrm D^{\tilde{\bm{\gamma}}}f^\ast\big\rVert_{\mathcal{C}(\M)}=\abs{f^\ast}_{\mathcal{C}^k(\M)}^\ast \le \abs{f^\ast}^\ast_{\mathcal{C}^{k,\alpha}(\M)}.
	\end{align*}
	By \eqref{eq:exact_bound}, this implies
	\begin{equation*}
		\abs{\mathrm D^{\bm{\gamma}}f^\ast \circ \phi}_{\mathcal{C}^\alpha(\T^d)}\le  \max \left\{ \sqrt{d}, \frac{2}{{d^\prime}}\right\} \, \frac{{d^\prime}!}{({d^\prime}-1)!} \, \abs{f^\ast}^\ast_{\mathcal{C}^{k,\alpha}(\M)} 
    ,
	\end{equation*}
	which shows \eqref{eq:H_bound}. As in the proof of \Cref{thm:hC_bound}, we apply \Cref{lem:faaDi} to conclude that $\tilde{f} \in \mathcal{C}^k(\T^d)$ and for $\bm{\beta} \in B_k^d$ we obtain 
	\begin{equation*}
		\mathrm D^{\bm{\beta}}\tilde{f}=\mathrm D^{\bm{\beta}}\pare{f^\ast \circ \phi}=\sum_{i=1}^n \pare{\mathrm D^{\bm{\gamma}_i}f^\ast \circ \phi} \, \prod_{j=1}^{m_i} \mathrm D^{\bm{\mu}_{ij}}\phi_{\ell_{ij}}
	\end{equation*}
	for some constants satisfying \eqref{eq:indH2} to \eqref{eq:indH6}. For $\bm{x},\bm{y}\in \R^d$, we apply the triangle inequality and get
	\begin{align*}
		\big\lvert \mathrm D^{\bm{\beta}}\tilde{f}(\bm{x})-\mathrm D^{\bm{\beta}}\tilde{f}(\bm{y})\big\rvert\le & \underbrace{\sum_{i=1}^n \big\lvert \pare{\mathrm D^{\bm{\gamma}_i}f^\ast \circ \phi}(\bm{x})-\pare{\mathrm D^{\bm{\gamma}_i}f^\ast \circ \phi}(\bm{y})\big\rvert \, \prod_{j=1}^{m_i} \abs{\mathrm D^{\bm{\mu}_{ij}}\phi_{\ell_{ij}}(\bm{x})}}_{\eqqcolon A}\\
		& +\underbrace{\sum_{i=1}^n \big\lvert \pare{\mathrm D^{\bm{\gamma}_i}f^\ast \circ \phi}(\bm{y})\big\rvert \, \Big\rvert\prod_{j=1}^{m_i} \mathrm D^{\bm{\mu}_{ij}}\phi_{\ell_{ij}}(\bm{x})-\prod_{j=1}^{m_i} \mathrm D^{\bm{\mu}_{ij}}\phi_{\ell_{ij}}(\bm{y})\Big\lvert}_{\eqqcolon B}.
	\end{align*}	
	The first sum can be estimated as
	\begin{alignat*}{3}
		& A && \centermathcell{\underset{\eqref{eq:DFSinf}}{\le}} && \sum_{i=1}^n \big\lvert \pare{\mathrm D^{\bm{\gamma}_i}f^* \circ \phi}(\bm{x}) - \pare{\mathrm D^{\bm{\gamma}_i}f^* \circ \phi}(\bm{y})\big\rvert 
		\underset{\eqref{eq:H_bound}}{\le} \sum_{i=1}^n \max \left\{ \sqrt{d} \, {d^\prime}, 2\right\} \, \abs{f^\ast}^\ast_{\mathcal{C}^{k,\alpha}(\M)} \norm{\bm{x}-\bm{y}}^\alpha\\
		& && \centermathcell{\underset{\eqref{eq:indH2}}{\le}} && \frac{(k+{d^\prime}-1)!}{({d^\prime}-1)!} \, \max \left\{ \sqrt{d} \, {d^\prime}, 2\right\} \abs{f^\ast}^\ast_{\mathcal{C}^{k,\alpha}(\M)}
    \norm{\bm{x}-\bm{y}}^\alpha.
	\end{alignat*}
	Furthermore, we use \eqref{eq:H_semi_manifold} to bound the second sum by 
	\begin{equation*}
		B \le \sum_{i=1}^n \abs{f^\ast}^\ast_{\mathcal{C}^{k,\alpha}(\M)} \, \Big\lvert \prod_{j=1}^{m_i} \mathrm D^{\bm{\mu}_{ij}}\phi_{\ell_{ij}}(\bm{x})-\prod_{j=1}^{m_i} \mathrm D^{\bm{\mu}_{ij}}\phi_{\ell_{ij}}(\bm{y})\Big\rvert.
	\end{equation*}
	By a telescoping sum, we rewrite the last sum as
	\begin{align*}
		B & \le \sum_{i=1}^n \abs{f^\ast}^\ast_{\mathcal{C}^{k,\alpha}(\M)}\, \Big\lvert\sum_{r=1}^{m_i} \big(\mathrm D^{\bm{\mu}_{ir}}\phi_{\ell_{ir}}(\bm{x})-\mathrm D^{\bm{\mu}_{ir}}\phi_{\ell_{ir}}(\bm{y})\big)  \prod_{j=1}^{r-1} \mathrm D^{\bm{\mu}_{ij}}\phi_{\ell_{ij}}(\bm{x}) \prod_{j=r+1}^{m_i} \mathrm D^{\bm{\mu}_{ij}}\phi_{\ell_{ij}}(\bm{y})\Big\rvert.
	\end{align*}
	Thus, we have
	\begin{alignat*}{3}
		& B && \centermathcell{\underset{\eqref{eq:DFSinf}}{\le}} && \sum_{i=1}^n \sum_{r=1}^{m_i} \abs{f^\ast}^\ast_{\mathcal{C}^{k,\alpha}(\M)} \, \abs{\mathrm D^{\bm{\mu}_{ir}}\phi_{\ell_{ir}}(\bm{x})-\mathrm D^{\bm{\mu}_{ir}}\phi_{\ell_{ir}}(\bm{y})}\\
		& && \centermathcell{\underset{\eqref{eq:DFSh}}{\le}} && \sum_{i=1}^n \sum_{r=1}^{m_i} \abs{f^\ast}^\ast_{\mathcal{C}^{k,\alpha}(\M)} \, 2 \sqrt{d} \, \norm{\bm{x}-\bm{y}}^\alpha.
	\end{alignat*}
	Finally, the bounds \eqref{eq:indH2} and \eqref{eq:indH4} yield
	\begin{equation*}
		B\le 
		\frac{(k+{d^\prime}-1)!}{({d^\prime}-1)!} \, 2k  \sqrt{d} \, \abs{f^\ast}^\ast_{\mathcal{C}^{k,\alpha}(\M)} \, \norm{\bm{x}-\bm{y}}^\alpha.
	\end{equation*}
    We combine the estimates on $A$ and $B$ to obtain
    \begin{align*}
        \big\lvert \mathrm D^{\bm{\beta}} \tilde{f}(\bm{x})-\mathrm D^{\bm{\beta}} \tilde{f}(\bm{y})\big\rvert & \le \frac{(k+{d^\prime}-1)!}{({d^\prime}-1)!} \, \pare{2k  \sqrt{d}+\max \left\{ \sqrt{d} \, {d^\prime}, 2\right\}}
          \abs{f^*}_{\mathcal{C}^{k,\alpha}(\M)}^* \, \norm{\bm{x}-\bm{y}}^\alpha\\
        & \le \frac{(k+{d^\prime})!}{({d^\prime}-1)!} \, 2\sqrt{d} \, 
          \abs{f^*}_{\mathcal{C}^{k,\alpha}(\M)}^* \, \norm{\bm{x}-\bm{y}}^\alpha    .
    \end{align*}
    The last equation holds for arbitrary $\bm{x},\bm{y} \in \R^d$ and $\bm{\beta} \in B^d_k$, thus we have
    \begin{equation*}
    	\big\lvert \tilde{f} \big\rvert_{\mathcal{C}^{k,\alpha}(\T^d)} \le 2  \sqrt{d} \, \frac{(k+{d^\prime})!}{({d^\prime}-1)!}  \, \abs{f^*}_{\mathcal{C}^{k,\alpha}(\M)}^*. 
    \end{equation*}
	Since this holds independently of the choice of $\mathcal{C}^{k,\alpha}$-extension $f^\ast$ of $f$, we can take the infimum over all such extensions. By the definition of the $\mathcal{C}^{k,\alpha}(\M)$-norm in \eqref{eq:H_norm_manifold}, this yields \eqref{eq:hH_bound}.
\end{proof}

%% file: 05_series_expansion.tex
\section{Series expansions with the DFS method}\label{sec:series}

We propose a series representation of functions with the generalized DFS method. In \Cref{sec:Fourier}, we explore properties of the Fourier series of DFS functions and define an analogue series expansion on the manifold.
In \Cref{sec:convergence}, we combine our findings from \Cref{thm:hC_bound,thm:hH_bound} with results from multi-dimensional Fourier analysis to show pointwise and uniform convergence of the Fourier series of DFS functions.

Throughout this section, let $\M \subset \R^{d^\prime}$ be an $d$-dimensional smooth embedded submanifold with or without corners that admits a generalized DFS transform $\phi \colon \T^d \to \M$. We write ${\tilde{f}=f \circ \phi \colon \T^d \to \C}$ for the generalized DFS function of a function $f \colon \M \to \C$.

\input{05.1_Fourier_series}
\input{05.2_convergence_results}

%% file: 05.1_Fourier_series.tex
\subsection{Fourier series and the DFS method}\label{sec:Fourier}

Let $L_2(\T^d)$ denote the Hilbert space of square-integrable complex-valued functions on $\T^d$ that is equipped with the inner product
\begin{equation*}
	\inn{g,h}_{L_2(\T^d)} \coloneqq (2\pi)^{-d} \int_{[-\pi,\pi]^d} g(\bm x) \, \overline{h(\bm x)} \d \bm x, \quad g,h \in L_2(\T^d)
\end{equation*}
and the induced norm $\norm{g}_{L_2(\T^d)} \coloneqq \inn{g,g}^{1/2}_{L^2(\T^d)}$. As the complex exponentials $\e^{\im \inn{\bm{n},\cdot}}$, $\bm{n}\in\Z^d$, form an orthonormal basis therein, we define the respective Fourier expansion.

\begin{defin}\label{defin:F_torus}
    Let $g \in L_2 (\T^d)$ and $\bm{n} \in \mathbb{Z}^d$. We define the \emph{$\bm{n}$-th Fourier coefficient} of $g$ by
    \begin{equation*}
        c_{\bm{n}}(g) \coloneqq \inn{g,\e^{\im \inn{\bm n,\cdot}}}_{L_2(\T^d)}=
        (2\pi)^{-d} \int_{[-\pi,\pi]^d} g(\bm{x})\, \e^{-\im \inn{\bm{n},\bm{x}}} \, \d\bm{x}.
    \end{equation*}
    Let $\Omega_h, \, h \in \N$, be an expanding sequence of bounded sets that exhausts $\Z^d$.
    We define the \emph{$h$-th partial Fourier sum} of $g$ by
    \begin{equation*}
        \mathrm F_{\Omega_h} g(\bm{x}) \coloneqq \sum_{\bm{n} \in \Omega_h} c_{\bm{n}}(g)\, \e^{\im \inn{\bm{n},\bm{x}}},\qquad \bm{x} \in \T^d,
    \end{equation*}
    and the \emph{Fourier series} of $g$ by $ \mathrm F g \coloneqq \lim_{h\to\infty} \mathrm F_{\Omega_h} g $. This limit is well-defined in $L_2(\T^d)$ for any choice of expanding sequence and we have $\mathrm Fg=g$ in $L_2(\T^d)$. We call a multi-series $\sum_{\bm{n}\in \Z^d} a_{\bm{n}}$ convergent whenever for all expanding sequences $\Omega_h, \, h\in \mathbb{N}$, of bounded sets exhausting $\Z^d$ the partial sums $\sum_{\bm{n} \in \Omega_h} a_{\bm{n}}$ converge as $h \to \infty$, cf.\ \cite[p.~6]{KhaNik92}.
\end{defin}

The generalized DFS method represents a function $f \colon \M \to \C$ via the Fourier series of its DFS function $\tilde f \colon \T^d \to \C$, i.e.,
\begin{equation}\label{eq:DFSseries1}
  \mathrm F \tilde f(\bm{x}) = 
  \sum_{\bm{n}\in\Z^d} c_{\bm{n}}(\tilde f)\, \e^{\im \inn{\bm{n},\bm x}}
  ,\qquad \bm x\in\T^d.
\end{equation} 
Ultimately, we are interested in representing the function $f$, not its DFS function $\tilde f$. 
Thus, the question whether we can relate the Fourier series \eqref{eq:DFSseries1} to a series defined on $\M$ arises naturally. Choosing $D_1$ and $D_2$ as in \Cref{defin:DFS},  \ref{Prop2}, we observe that the restriction of $\phi$ to $D_1 \cup D_2$ is bijective, so we could just apply its inverse to the basis functions $\e^{\im \inn{\bm n,\cdot}}$, $\bm n \in \Z^d$. However, this would yield some redundancies in the expansion
since $\e^{\im \inn{\bm n,\cdot}}$ is in general not a BMC function, cf.\ \Cref{lem:BMC}.
We account for this by defining an orthogonal basis of BMC functions in $L^2(\T^d)$ consisting of suitable linear combinations of the functions $\e^{\im \inn{\bm n,\cdot}}$. To this end, we will also require a suitable subset of the indices in $\Z^d$.
This way, we can obtain a respective basis on the manifold $\M$ in the following.

From \Cref{defin:DFS}, we recall the symmetry number $p$, the shift vectors $\mathrm S^i$, the reflection maps $\mathrm M^i$, and the symmetry functions $s^i$ defined in \eqref{eq:si}. Note that the functions $\mathrm M^i$ are well-defined on $\Z^d$ by the matrix multiplication in \eqref{eq:Mi}. For $\Omega \subset \Z^{d}$, we set
\begin{equation*}
	\mathcal M(\Omega)\coloneqq \left\{\bm n \in \Z^d \mid \mathrm M^I(\bm n)=\pare{\circ_{i \in I} \mathrm M^i}(\bm n) \in \Omega \text{ for some } I \subset [p] \right\},
\end{equation*}
and $\mathcal M(\bm n) \coloneqq \mathcal M(\{\bm n\})$ for $\bm n\in\Z^d$. For $I\subset [p]$, we introduce the map
\begin{equation*}
	\mathrm{N}^I \colon \Z^d \to \Z, \, \mathrm N^I(\bm n) \coloneqq \sum_{i\in I} \sum_{j \in [d], \, \mathrm S^i_j=\pi} n_j
\end{equation*}
and we write
$\mathrm{N}^i(\bm n) \coloneqq \mathrm{N}^{\{i\}}(\bm n)$. For $\bm n \in \Z^d$, we define 
\begin{equation}\label{eq:r_n,n}
	r_{\bm n,\bm n} \coloneqq \begin{cases}
		1, & \mathrm N^I(\bm n) \text{ is even for all } I \subset [p] \text{ with } \mathrm M^I(\bm n)=\bm n\\
		0, &  \mathrm N^I(\bm n) \text{ is odd for some } I \subset [p] \text{ with } \mathrm M^I(\bm n)=\bm n.
	\end{cases}
\end{equation}
Note that if the reflections $\mathrm M^i$ act on pairwise disjoint sets of variables, then for any $I \subset [p]$, we have $\mathrm M^I(\bm n)=\bm n$ if and only if $\mathrm M^i(\bm n)=\bm n$ for all $i \in I$. In this situation, we therefore only need to check the parity of $\mathrm N^i(\bm n)$ for all $i \in [p]$ to determine $r_{\bm n,\bm n}$. For $J \subset [p]$ and $\bm m \coloneqq \mathrm M^J(\bm n)$, we define
\begin{equation}\label{eq:r_n,m}
	r_{\bm n,\bm m} \coloneqq (-1)^{\mathrm N^J(\bm n)} \, r_{\bm n,\bm n}.
\end{equation}
The next lemma will show that $r_{\bm n,\bm m}$ is thus well-defined for $\bm m \in \mathcal{M}(\bm n)$.  For $I, J \subset [p]$, we use $\#I$ to denote the cardinality of $I$ and we define the symmetric difference
\begin{equation*}
	I \Delta J \coloneqq \pare{I \setminus J} \dot{\cup} \pare{J \setminus I}.
\end{equation*}

\begin{lem}
	Let $\bm n \in \Z^d$ and $\bm m \in \mathcal{M}(\bm n)$. It holds that
	\begin{equation}\label{eq:cardinal}
    	\bm n_{\#}\coloneqq \#\{I \subset [p]\mid \mathrm M^I(\bm n)=\bm n\}=\#\{I \subset [p]\mid \mathrm M^I(\bm n)=\bm m\}
	\end{equation}
	and we have
	\begin{equation}\label{eq:r_formula}
		r_{\bm n,\bm m}= \frac{1}{\bm n_\#} \sum_{\substack{I \subset [p],\\ \mathrm M^{I}(\bm n)=\bm m}} (-1)^{\mathrm N^I(\bm n)}
    .
	\end{equation}
	In particular $r_{\bm n,\bm m}$ is well-defined by \eqref{eq:r_n,n} and \eqref{eq:r_n,m}. Furthermore, we have for $I,J \subset [p]$ that
	\begin{equation}\label{eq:N_sum}
			(-1)^{\mathrm N^I(\bm n)+\mathrm N^J(\bm n)}=(-1)^{\mathrm N^{I\Delta J}(\bm n)}
	\end{equation}
	and
  	\begin{equation}\label{eq:M_sum}
  		\mathrm M^I \left(\mathrm M^J(\bm n)\right) = \mathrm M^{I\Delta J}(\bm n).
  	\end{equation}
\end{lem}
\begin{proof}
	For $I \subset [p]$ and $\ell \in [p]$, we have 
	\begin{align*}
		(-1)^{\mathrm N^I(\bm n)+\mathrm N^\ell(\bm n)} & =(-1)^{\mathrm N^I(\bm n)\pm \mathrm N^\ell(\bm n)} =\begin{cases}
				(-1)^{\mathrm N^{I \cup \{\ell\}}(\bm n)}, & \ell \not\in I\\
				(-1)^{\mathrm N^{I \setminus \{\ell\}}(\bm n)}, & \ell \in I.
			\end{cases}
	\end{align*}
	This inductively yields \eqref{eq:N_sum}. Analogously, we obtain \eqref{eq:M_sum} from
	\begin{equation*}
  		\mathrm M^I \left(\mathrm M^\ell(\bm n)\right) =\begin{cases} \pare{\circ_{i \in I \cup \{\ell\}}\mathrm M^i}(\bm n), & \ell \not\in I\\    \pare{\circ_{i\in I \setminus \{\ell\}} \mathrm M^i}(\bm n), & \ell \in I,\end{cases}
  	\end{equation*}
  	where we used that the reflections $\mathrm M^i$ 
  	are self-inverse and commute. By \eqref{eq:M_sum}, we observe that the map
	\begin{align*}
		\{I \subset [p] \mid \mathrm M^I(\bm n)=\bm n\} &\to \{I \subset [p] \mid \mathrm M^I(\bm n)=\bm m\}\\
		I & \mapsto I \Delta J
	\end{align*}
	is self-inverse. In particular, this proves \eqref{eq:cardinal}.
 
  	Next, we show \eqref{eq:r_formula} for the case $\bm m=\bm n$. If $\mathrm N^J(\bm n)$ is odd for some $J \subset [p]$ with $\mathrm N^J(\bm n)=\bm n$, then
  	\begin{alignat*}{3}
		& \frac{1}{\bm n_{\#}}\sum_{\substack{I \subset [p],\\ \mathrm M^I(\bm n)=\bm n}}(-1)^{\mathrm N^I(\bm n)}
		&&\centermathcell{=} && \frac{1}{2 \bm n_{\#}}\sum_{\substack{I \subset [p],\\ \mathrm M^I(\bm n)=\bm n}}(-1)^{\mathrm N^I(\bm n)} + (-1)^{\mathrm N^{I\Delta J}(\bm n)}\\
		& &&\centermathcell{\underset{\eqref{eq:N_sum}}{=} \,} && \frac{1}{2\bm n_{\#}} \sum_{\substack{I \subset [p],\\ \mathrm M^I(\bm n)=\bm n}}\underbrace{(1+(-1)^{\mathrm N^J(\bm n)})}_{=0}(-1)^{\mathrm N^I(\bm n)}=0 \underset{\eqref{eq:r_n,n}}{=} r_{\bm n,\bm n}.
	\end{alignat*}
	On the other hand, if $\mathrm N^I(\bm n)$ is even for all $I\subset [p]$ with $\mathrm M^I(\bm n)=\bm n$, then 
	\begin{equation*}
		\frac{1}{\bm n_{\#}} \sum_{\substack{I \subset [p],\\ \mathrm M^I(\bm n)=\bm n}} \underbrace{(-1)^{\mathrm N^I(\bm n)}}_{=1}=\frac{\bm n_{\#}}{\bm n_{\#}}=1 \underset{\eqref{eq:r_n,n}}{=} r_{\bm n,\bm n}.
	\end{equation*}
  	Now, let $\bm m \in \mathcal{M}(\bm n)$ be arbitrary. For $J \subset [p]$ with $\mathrm M^J(\bm n)=\bm m$, we obtain \eqref{eq:r_formula} by
  	\begin{alignat*}{3}
		& (-1)^{\mathrm N^J(\bm n)} \, r_{\bm n,\bm n} && \centermathcell{=} && (-1)^{\mathrm N^J(\bm n)} \frac{1}{\bm n_{\#}} \sum_{\substack{I \subset [p],\\ \mathrm M^I(\bm n)=\bm n}} (-1)^{\mathrm N^I(\bm n)}\\
		& && \centermathcell{\underset{\eqref{eq:N_sum}}{=}\, } && \frac{1}{\bm n_{\#}}\sum_{\substack{I \subset [p],\\ \mathrm M^I(\bm n)=\bm n}}(-1)^{\mathrm N^{I\Delta J}(\bm n)} = \frac{1}{\bm n_{\#}}\sum_{\substack{I \subset [p],\\ \mathrm M^I(\bm n)=\bm m}}(-1)^{\mathrm N^I(\bm n)}.
	\end{alignat*}
  	Note that the right hand side is independent of $J$. This proves that $r_{\bm n,\bm m}$ is well-defined.
\end{proof}

To construct a BMC basis, we choose an index set $\Omega_\phi \subset \Z^d$ that satisfies both 
\begin{equation}\label{eq:redundancy}
	\mathcal M(\bm n)\cap \mathcal M(\bm m)=\emptyset \text{ for all } \bm n,\bm m \in \Omega_\phi \text{ with }\bm n \not=\bm m
\end{equation} 
and
\begin{equation}\label{eq:non-zero}
	\mathcal M(\Omega_\phi)=\{\bm n \in \Z^d\mid r_{\bm n,\bm n}\not=0\}.
\end{equation}
The next theorem {shows the existence of} such an $\Omega_\phi$ and gives a BMC basis.

\begin{thm}\label{lem:Fourier_series}
	There exists a set $\Omega_\phi \subset \Z^d$ that fulfills \eqref{eq:redundancy} and \eqref{eq:non-zero}. For $\bm n \in \Omega_\phi$, the function
	\begin{equation}\label{eq:en}
		e_{\bm n}(\bm x)\coloneqq \sum_{\bm m \in \mathcal M(\bm n)} r_{\bm n,\bm m} \, \e^{\im \inn{\bm m,\bm x}}, \qquad \bm x \in \T^d,
	\end{equation}
	is a BMC function with
	\begin{equation*}
		\norm{e_{\bm n}}_{L_2(\T^d)}^2=
			\# \mathcal M(\bm n)\le 2^p. 
	\end{equation*}
	Furthermore, the family $e_{\bm n}$, $\bm n \in \Omega_\phi$, is an orthogonal basis of the subspace of BMC functions in $L_2(\T^d)$. For any BMC function $g \in L_2(\T^d)$ and any finite $\Omega \subset \Omega_\phi$, we have
	\begin{equation}\label{eq:series_sym}
		\mathrm F_{\mathcal M(\Omega)}g(\bm x)=\sum_{\bm n \in \mathcal M(\Omega)} c_{\bm n}(g) \, \e^{\im \inn{\bm n,\bm x}}=\sum_{\bm n \in \Omega} c_{\bm n}(g) \, e_{\bm n}(\bm x), \quad \bm x \in \T^d.
	\end{equation}
\end{thm}

\begin{proof}
	We first construct an $\Omega_\phi$ that fulfills \eqref{eq:redundancy} and \eqref{eq:non-zero}. We consider the equivalence relation
	$
		\bm n \sim \bm m \Leftrightarrow \mathcal M(\bm n) = \mathcal M(\bm m)
	$ on the finite set $\{-2,-1,0,1,2\}^d$ and choose some set $\Omega$ of representatives of the corresponding quotient space. We then define
	\begin{equation*}
		\Omega^\prime \coloneqq \{\bm n \in \Omega \mid r_{\bm n,\bm n}\not=0\}
	\end{equation*}
	and 
	\begin{equation*}
		\Omega_\phi \coloneqq \bigcup_{\bm n \in \Omega^\prime} \pare{\bigtimes_{j=1}^d \big(n_j+2\sgn(n_j) \, \NN\big)},
	\end{equation*}
  where $\bigtimes$ denotes the Cartesian product of sets. For $\bm n,\bm m \in \Z^d$, we have $\mathcal M(\bm n)=\mathcal M(\bm m)$ or $\mathcal M(\bm n)\cap \mathcal M(\bm m)=\emptyset$. Since the reflection maps $\mathrm M^i$, $i \in [p]$, only change signs, we have  $\mathcal M(\bm n)=\mathcal M(\bm m)$ if and only if $\abs{n_j}=\abs{m_j}$ for all $j \in [m]$ and $\mathcal M(\tilde{\bm n})=\mathcal M(\tilde{\bm m})$, where $\tilde{n}_j=\sgn(n_j)$ and $\tilde{m}_j=\sgn(m_j)$. Thus, $\Omega_\phi$ satisfies \eqref{eq:redundancy}. Furthermore, we observe that the value of $r_{\bm n,\bm n}$ only depends on which components of $\bm n$ are zero and which components are odd. In particular, $r_{\bm m,\bm m}=r_{\bm n,\bm n}$ for all $\bm m \in \mathcal M(\bm n)$ and adding even integers to $\Omega^\prime$ ensures \eqref{eq:non-zero}.
	
	We now show that for all $\bm n \in \Z^d$ the function $e_{\bm n}$, as defined by \eqref{eq:en}, is a BMC function. By \eqref{eq:r_formula} and the definition of $\mathcal M(\bm n)$, we can rewrite $e_{\bm n}$ as
	\begin{equation}\label{eq:basis_sym}
		e_{\bm{n}}  = \sum_{\bm m \in \mathcal M(\bm n)}\frac{1}{\bm n_{\#}} \sum_{\substack{I \subset [p],\\ \mathrm M^I(\bm n)=\bm m}}(-1)^{\mathrm N^I(\bm n)} \, \e^{\im \inn{\bm m,\cdot}}= \frac{1}{\bm n_\#} \sum_{I \subset [p]} (-1)^{\mathrm N^I(\bm n)} \, \e^{\im \inn{\mathrm M^I(\bm n),\cdot}}.
	\end{equation}	
	Let $\ell \in [p]$ and $\bm x \in \T^d$. For the symmetry function $s^\ell$ from \Cref{defin:DFS}, we have
  \begin{align}\label{eq:exp_id}
    \begin{split}
      \e^{\im \inn{\bm n,s^\ell(\bm x)}} 
      &= \e^{\im \inn{\bm n,\mathrm S^\ell+\mathrm M^\ell(\bm x)}}= \e^{\im \inn{\bm n,\mathrm S^\ell}}\, \e^{\im \inn{\bm n,\mathrm M^\ell(\bm x)}}\\
      &= \e^{\im \pi \sum_{j\in[d],\mathrm S^\ell_j=\pi} n_j}\, \e^{\im \inn{\mathrm M^\ell(\bm n),\bm x}}
      =(-1)^{\mathrm N^\ell(\bm n)}\e^{\im \inn{\mathrm M^\ell(\bm n),\bm x}}.
    \end{split}	
  \end{align}
	For $I \subset [p]$ it clearly holds that $\pare{\mathrm M^i(\bm n)}_j=\pm n_j$ and thus
	\begin{equation}\label{eq:minus_exp}
		(-1)^{\mathrm N^\ell \pare{\mathrm M^I(\bm n)}}=(-1)^{\sum_{j\in [d], \mathrm S^\ell_j=\pi}\pm n_j}=(-1)^{\sum_{j \in [d],\mathrm S^\ell_j=\pi} n_j}=(-1)^{\mathrm N^\ell(\bm n)}.
	\end{equation}
	Altogether, using the definition of $\mathcal M(\bm n)$, this implies
	\begin{alignat*}{3}
		& e_{\bm n}(s^\ell(\bm x)) && \centermathcell{\underset{\eqref{eq:basis_sym}}{=}} && \frac{1}{\bm n_\#} \sum_{I \subset [p]} (-1)^{\mathrm N^I(\bm n)} \, \e^{\im  \inn{\mathrm M^I(\bm n),s^\ell(\bm x)}}\\
		& && \centermathcell{\underset{\eqref{eq:exp_id}}{=}} &&\frac{1}{\bm n_\#} \sum_{I \subset [p]} (-1)^{\mathrm N^I(\bm n)}(-1)^{\mathrm N^\ell\pare{\mathrm M^I(\bm n)}} \, \e^{\im  \inn{\mathrm M^\ell \pare{\mathrm M^I(\bm n)},\bm x}}\\
		& &&\centermathcell{\underset{\eqref{eq:minus_exp}}{=}} &&\frac{1}{\bm n_\#}\sum_{I \subset [p]} (-1)^{ \mathrm N^I(\bm n)+\mathrm N^\ell(\bm n)} \, \e^{\im \inn{\mathrm M^\ell \pare{\mathrm M^I(\bm n)},\bm x}}\\
		 & &&\centermathcell{\underset{\eqref{eq:N_sum},\eqref{eq:M_sum}}{=}\,} &&\frac{1}{\bm n_\#} \sum_{I \subset [p]} (-1)^{\mathrm N^{I \Delta \{\ell\}}(\bm n)} \, \e^{\im \inn{\mathrm M^{I \Delta \{\ell\}}(\bm n),\bm x}}\underset{\eqref{eq:basis_sym}}{=}e_{\bm n}(\bm x),
	\end{alignat*}
  which proves that $e_{\bm n}$ is a BMC function.
  
   The orthonormality of the Fourier basis $\e^{\im \inn{\bm n,\cdot}}$, $\bm n \in \Z^d$, combined with \eqref{eq:non-zero} immediately implies for all $\bm n \in \Omega_\phi$ that
	\begin{align*}
		\norm{e_{\bm n}}^2_{L_2(\T^d)}
		& =\sum_{\bm m \in \mathcal M(\bm n)} \underbrace{r_{\bm n,\bm m}^2}_{= \, 1}=\#\mathcal M(\bm n)\le \#\mathcal{P}([p])=2^p,
	\end{align*}
	where $\mathcal{P}$ denotes the power set.
	
	Next, we show the orthogonality of $e_{\bm n}$, $\bm n \in \Omega_\phi$. Let $\bm n,\bm m \in \Omega_\phi$ with $\bm n\not=\bm m$. Then, by \eqref{eq:redundancy}, we have $\mathcal M(\bm n)\cap \mathcal M(\bm m)=\emptyset$.
  Hence, for any $I,J \subset [p]$, it holds that
	${\mathrm M^I(\bm n)\not=\mathrm M^J(\bm m)}$,
  which implies
	 \begin{equation*}
	 	\inn{\e^{\im \inn{\mathrm M^I(\bm n),\cdot}},\e^{\im \inn{ \mathrm M^J(\bm m),\cdot}}}_{L_2(\T^d)}=0, \qquad I,J \subset [p].
	 \end{equation*} 
	 The bilinearity of the inner product and \eqref{eq:basis_sym} now immediately yield $\inn{e_{\bm n},e_{\bm m}}_{L_2(\T^d)}=0$. 
	 
	Let $g \in L_2(\T^d)$ be a BMC function and $i \in [p]$. We employ \eqref{eq:exp_id}, properties of $s^i$, and a change of variables to obtain that for any $\bm n \in \Z^d$
	\begin{align*}
		(2\pi)^d \, c_{\bm n}(g) &= \int_{\T^d} g(\bm x) \, \e^{-\im\inn{\bm n,\bm x}} \d \bm x= \int_{\T^d} g(s^i(\bm x)) \, \e^{-\im \inn{\bm n,\bm x}} \d \bm x\\
		& = \int_{\T^d} g(\bm x) \, \e^{-\im \inn{\bm n,s^i(\bm x)}}\d \bm x = (-1)^{\mathrm N^i(\bm n)} \int_{\T^d} g(\bm x) \, \e^{-\im \inn{\mathrm M^i(\bm n),\bm x}}\d \bm x.
	\end{align*}
	Division by $(2\pi)^d$ and induction yields
	\begin{equation}\label{eq:coeff_sym}
		c_{\bm n}(g)=(-1)^{\mathrm N^I(\bm n)} \, c_{\mathrm M^I(\bm n)}(g), \quad \bm n \in \Z^d, I\subset [p].
	\end{equation}
 We can now show \eqref{eq:series_sym}. Consider some finite set $\Omega \subset \Omega_\phi$. It is clear from \eqref{eq:redundancy} and the definition of $\mathcal M(\Omega)$ that 
	 \begin{equation*}
	 	\mathcal M(\Omega)=\dot{\bigcup_{\bm n \in \Omega}} \mathcal M(\bm n),
	 \end{equation*}
	 where $\dot{\cup}$ stands for a disjoint union. For $\bm n \in \Omega$, we have
	 \begin{alignat*}{3}
	 	& \sum_{\bm m \in \mathcal M(\bm n)} c_{\bm m}(g) \, \e^{\im \inn{\bm m,\cdot}} && \centermathcell{\underset{\eqref{eq:cardinal}}{=}} && \frac{1}{\bm n_\#}\sum_{I \subset [p]} c_{\mathrm M^I(\bm n)}(g) \, \e^{\im \inn{\mathrm M^I(\bm n),\cdot}}\\
	 	& && \centermathcell{\underset{\eqref{eq:coeff_sym}}{=}\,} && \frac{1}{\bm n_\#} \sum_{I \subset [p]} (-1)^{\mathrm N^I(\bm n)} \, c_{\bm n}(g) \, \e^{\im \inn{\mathrm M^I(\bm n),\cdot}}
	 	{\underset{\eqref{eq:basis_sym}}{=}}  c_{\bm n}(g)\,  e_{\bm n}.
	 \end{alignat*}
	 This shows
	 \begin{equation*}
	 	 \mathrm F_{\mathcal M(\Omega)} g=\sum_{\bm n \in \Omega} c_{\bm n}(g) \, e_{\bm n}
	 \end{equation*}
	 and in particular implies \eqref{eq:series_sym}.
	 
   We show the completeness of $e_{\bm n}$, $\bm n \in \Omega_\phi$.
   Let $g \in L_2(\T^d)$ be a BMC function such that $\inn{g, e_{\bm n}}=0$ for all $\bm n\in\Omega_\phi$.
   We need to show that $c_{\bm m}(g)=0$ for all $\bm m\in\Z^d$, which implies $g=0$.
   Let $\bm m\in\Z^d$. 
   If $r_{\bm m,\bm m}=0$, then there exists $I\subset [p]$ with $\mathrm N^I(\bm m)$ odd and $\mathrm M^I(\bm m)=\bm m$.
   By \eqref{eq:coeff_sym}, we have $c_{\bm m}(g) = - c_{\bm m}(g)$ and hence $c_{\bm m}(g) = 0$.
   If $r_{\bm m,\bm m}\neq0$, by \eqref{eq:non-zero}, there exists $\bm n\in\Omega_\phi$ and $J\subset [p]$ such that $\mathrm M^J(\bm n)=\bm m$ and we obtain
   \begin{alignat*}{3}
     & 0
     = \inn{g,e_{\bm n}}
     && \centermathcell{\underset{\eqref{eq:basis_sym}}{=}} && \frac{1}{\bm n_\#} \sum_{I \subset [p]} (-1)^{\mathrm N^I(\bm n)}\, c_{\mathrm M^I(\bm n)}(g)\\
     & && \centermathcell{\underset{\eqref{eq:cardinal},\eqref{eq:coeff_sym}}{=}} && \# \mathcal{M}(\bm n) \, c_{\bm n}(g)
     \underset{\eqref{eq:coeff_sym}}{=} \underbrace{\# \mathcal{M}(\bm n) \, (-1)^{\mathrm N^J(\bm n)}}_{\not= \, 0} \, c_{\bm m}(g),
   \end{alignat*}
   which proves completeness.
\end{proof}

\begin{rem}
    Instead of using a standard FFT for the expansion in the basis $e_{\bm n}$, symmetry-dependent FFT variants can be used and thereby reduce computational cost. In one dimension, the discrete cosine transform (DCT) and discrete sine transform (DST) \cite{BriYipRao06} are well-known replacements of the FFT for real-valued functions with even or odd symmetries, respectively. Similarly, computational cost can be reduced if a Fourier series consists only of even or odd degrees \cite{Rab79}. Symmetry-dependent FFT variants also exist in higher dimensions, see e.g., \cite{Ber99}. Such algorithms often consist of combining one-dimensional techniques \cite{DevLab82,SerJag97}, for example by row-column methods \cite[ch.~5.3.5]{PloPotSteTas18}.
\end{rem}

The  functions $e_{\bm n}$ in the last theorem are BMC functions and thus, by \Cref{lem:BMC}, they correspond to functions defined on the manifold $\M$ without the null set of singularities. We use this correspondence to define an analogue to the Fourier series on $\M$. 

\begin{defin}\label{defin:F_manifold} 
	For $\bm n \in \Omega_\phi$, we define the \emph{$\bm n$-th DFS basis function}
	\begin{equation}\label{eq:bn}
		b_{\bm{n}}(\bm \xi) \coloneqq e_{\bm{n}} \Big( \big(\restr{\phi}{D_1\dot{\cup} D_2}\big)^{-1}(\bm \xi)\Big), \quad \bm \xi \in \M,
	\end{equation}
	with $e_{\bm n}$ as in \Cref{lem:Fourier_series}. Let $\Omega_h, \, h \in \N,$ be an expanding sequence of bounded sets exhausting $\Omega_\phi$. For $f \colon \M \to \C$ with $\tilde f \in L_2(\T^d)$, we define the \emph{$h$-th partial DFS Fourier sum} of $f$ by
	\begin{equation*}
		\mathrm S_{\Omega_h}f(\bm \xi) \coloneqq \sum_{\bm n \in \Omega_h} c_{\bm n}(\tilde f) \, b_{\bm n}(\bm \xi),\quad \bm \xi \in \M,
	\end{equation*}
	and the \emph{DFS Fourier series} of $f$ by $\mathrm S f\coloneqq \lim_{h \to \infty} \mathrm S_{\Omega_h}f$. 
\end{defin}
The basis functions $b_{\bm n}$ might be non-smooth on the set of singularities $\phi[D_2]$.
Regardless, we will later show that $\mathrm Sf$ converges uniformly on $\M$ provided that $f$ is sufficiently smooth. 

\begin{thm}\label{thm:manifold_DFS}
	Let $f \colon \M \to \C$ such that $\tilde{f} \in L_2(\T^d)$, and let $\Omega \subset \Omega_\phi$ be finite. Then, the partial DFS Fourier sum $\mathrm S_\Omega f$ is given by 
	\begin{equation}\label{eq:manifold_series_eq}
		 \mathrm S_{\Omega}f(\bm{\xi})=\mathrm F_{\mathcal M(\Omega)}\tilde{f}\Big( \big(\restr{\phi}{D_1\dot{\cup} D_2}\big)^{-1}(\bm \xi)\Big), \qquad \bm \xi \in \M
	\end{equation}
	and it is the unique function on $\M$ that satisfies
	\begin{equation}\label{eq:DFSseries_eq}
		\big((\mathrm S_{\Omega}f) \circ \phi\big)(\bm x) \eqqcolon \widetilde{\mathrm S_\Omega f}(\bm x)=\mathrm F_{\mathcal M(\Omega)}\tilde f(\bm x), \qquad \bm x \in \bigcup_{I \subset [p]} s^I[D_1] \cup D_2.
\end{equation}
In particular, this equality holds almost everywhere on $\T^d$. Furthermore, $\mathrm S_{\Omega}f$ is continuous on $\M \setminus \phi[D_2]$ and smooth in $\M^\circ\setminus \phi[D_2]$, where $\M^\circ$ denotes the manifold interior. If $\mathrm F \tilde f$ is pointwise convergent to $\tilde f$, then $\mathrm Sf$ is pointwise convergent to $f$.
\end{thm}

\begin{proof}
	By \Cref{lem:BMC}, the DFS function $\tilde f$ of $f$ is a BMC function. Thus, we obtain by \Cref{lem:Fourier_series} that
	\begin{equation*}
		\mathrm F_{\mathcal M(\Omega)}\tilde f(\bm x)=\sum_{\bm n \in \Omega} c_{\bm n}(\tilde f) \, e_{\bm n}(\bm x), \qquad \bm x \in \T^d.
	\end{equation*}
	Applying the inverse of ${\restr{\phi}{D_1\dot{\cup} D_2}}$ to both sides of the equation immediately yields \eqref{eq:manifold_series_eq} by the definition of $b_{\bm n}$ in \eqref{eq:bn}. Since $\mathrm F_{\mathcal M(\Omega)}\tilde f$ is the finite sum of BMC functions, and thus a BMC function, we can apply \Cref{lem:BMC} again to prove \eqref{eq:DFSseries_eq}. The regularity follows immediately form \Cref{lem:even}. We observe that for any expanding sequence $\Omega_h$, $h \in \N$, of bounded sets exhausting $\Omega_\phi$, the sequence $\mathcal M(\Omega_h)$, $h \in \N$, is an expanding sequence of bounded sets exhausting $\{\bm n \in \Z^d\mid r_{\bm n,\bm n}\not=0\}$. Since $\tilde{f}$ is a BMC function, we have  $c_{\bm n}(\tilde f) = 0$ for any $\bm n \in \Z^d$ with $r_{\bm n,\bm n}=0$ and thus the pointwise convergence follows with \eqref{eq:series_sym} and \eqref{eq:manifold_series_eq}.
\end{proof}

\begin{thm}\label{thm:weighted}
  We set $g(\bm x)\coloneqq \abs{\mathrm{det}\big(\nabla \phi(\bm x)^\intercal \nabla \phi(\bm x)\big)}^{1/2}$ for $\bm x\in D_1^{\circ}$. Let $\tilde{L}_2(\M)$ denote the $L_2$-space induced by the weighted inner product
  \begin{equation*}
    \inn{f_1,f_2}_{\tilde{L}_2(\M)}
    \coloneqq
    \int_{\M} f_1(\bm \xi) \, \overline{f_2(\bm \xi)} \, \frac{2^{p-d}\pi^{-d}}{{\big(g \circ (\restr{\phi}{D_1^\circ})^{-1}\big)(\bm \xi)}} \d \omega(\bm \xi),
    \qquad f_1,f_2\colon\M\to\C,
  \end{equation*}
  where $\mathrm{d}\omega$ denotes the surface measure on $\M$.
  The set $b_{\bm n}$, $\bm n \in \Omega_\phi$, forms an orthogonal basis of $\tilde{L}_2(\M)$ and 
  we have for all $f\in\tilde L_2(\M)$ that
  \begin{equation} \label{eq:manifold_c}
    c_{\bm n}(\tilde{f})=\frac{\inn{\smash{\tilde{f}},e_{\bm n}}_{L_2(\T^d)}}{\norm{e_{\bm n}}^2_{L_2(\T^d)}}=\frac{\inn{f,b_{\bm n}}_{\tilde{L}_2(\M)}}{\norm{b_{\bm n}}^2_{\tilde{L}_2(\M)}}
    ,\qquad \bm n \in \Omega_\phi.
  \end{equation}
  In particular, the DFS Fourier series $Sf$ of any $f \in \tilde{L}_2(\M)$ is convergent in $\tilde{L}_2(\M)$ and we have $Sf=f$.
\end{thm}
\begin{proof}
	By \Cref{lem:even}, the set $\M \setminus \phi[D_1^\circ]$ has measure zero in $\M$ and by \ref{Prop4} in \Cref{defin:DFS}, we have $g(\bm x)\not=0$ for all $\bm x \in D_1^\circ$. Hence, we can apply the substitution rule \cite[prop.~15.31]{Lee12} for the orientable submanifold $\phi[D_1^\circ]\subset\M$
  and we obtain for $f_1,f_2\in \tilde L_2(\M)$ the isometry
  \begin{align*}
    \langle \tilde{f}_1,\tilde{f}_2 \rangle_{L_2(\T^d)} & =(2\pi)^{-d} \int_{\T^d} \tilde{f}_1(\bm x) \, \overline{\tilde{f}_2(\bm x)} \d \bm x =  (2\pi)^{-d} \sum_{I \subset [p]} \int_{s^I[D_1^\circ]} \tilde{f}_1(\bm x) \, \overline{\tilde{f}_2(\bm x)} \d \bm x\\
    &  = 2^{p-d}\pi^{-d} \int_{D_1^\circ} \frac{f_1\big(\phi(\bm x)\big) \, \overline{f_2\big(\phi(\bm x)\big)}}{\big(g \circ (\restr{\phi}{D_1^\circ})^{-1}\big)\big(\phi(\bm x)\big)} \, g(\bm x) \d \bm x
    = \inn{f_1,f_2}_{\tilde{L}_2(\M)},
  \end{align*}
  where we used \eqref{eq:even} in the second equality and the $s^I$--invariance of DFS functions in the third equality.
  Since $\tilde{b}_{\bm n}=e_{\bm n}$ almost everywhere in $\T^d$, we obtain \eqref{eq:manifold_c}. 
\end{proof}

%% file: 05.2_convergence_results.tex
\subsection{Convergence of the Fourier series}\label{sec:convergence}

In this subsection, we finally obtain convergence results for DFS Fourier series on the manifold $\M$. Here, we closely follow the derivation in \cite[§~5.2]{MilQue22} and generalize it to higher dimensions. For brevity of notation, we set $\mathcal{C}^{k,1}(\M) := \mathcal{C}^{k+1}(\M)$ for $k\in\NN$, and the same for $\M$ replaced by $\T^d$.

\begin{lem}\label{lem:bSeries}
    Let $k \in \NN, \ 0<\alpha \leq 1$, and $f\in \mathcal{C}^{k,\alpha}(\M)$. Let $c_{\bm{n}}(\tilde{f}), \, \bm n \in \Z^d,$ denote the Fourier coefficients of the generalized DFS function $\tilde f$ of $f$. Then, the series
    \begin{equation}\label{eq:bSeries}
        \sum_{\bm{n} \in \Z^d} \big\lvert c_{\bm{n}}(\tilde{f})\big\rvert^b
    \end{equation}
    converges for all $b \in \mathbb{R}$ with $b > 2 d / \big({d+2\pare{k+\alpha}}\big)$.
\end{lem}

\begin{proof}
  If $\alpha<1$, we can apply \Cref{thm:hH_bound} to obtain $\tilde{f} \in \mathcal{C}^{k,\alpha}(\T^d)$.
  By \cite[p.~87]{KhaNik92}, this immediately implies the convergence of \eqref{eq:bSeries} for all $b> 2 d/\big({d+2\pare{k+\alpha}}\big)$. 
  If $\alpha=1$, we choose $0<\varepsilon<1$ such that $b> 2 d/\big({d+2\pare{k+1-\varepsilon}}\big)$. 
  \Cref{thm:hC_bound} then yields $\tilde{f}\in \mathcal{C}^{k,1-\varepsilon}(\T^d)$ and the convergence of \eqref{eq:bSeries} follows as before.
\end{proof}

\begin{thm}\label{thm:hoelScon}
    Let $k \in \NN$ and $0 < \alpha \leq 1$ such that $2\pare{k+\alpha} >d$. For $f\in \mathcal{C}^{k,\alpha}(\M)$ the Fourier series $\mathrm F \tilde{f}$ converges uniformly to the DFS function $\tilde{f}$ and for $\Omega\subset\Z^d$, it holds that
    \begin{equation*}
        \big\lVert\tilde{f} - \mathrm F_\Omega \tilde{f}\big\rVert_{\mathcal{C}(\T^d)} 
        \le \sum_{\bm{n} \in \Z^d\setminus \Omega} \big\lvert c_{\bm{n}}(\tilde{f})\big\rvert.
    \end{equation*}
    Furthermore, the DFS Fourier series $\mathrm S f$ converges to $f$ uniformly on $\M$. For $\Omega \subset \Omega_\phi$, we obtain
    \begin{equation*}
    	\big\lVert f-\mathrm S_{\Omega}f \big\rVert_{\mathcal{C}(\M)} \le \sum_{\bm n \in \Z^d\setminus \mathcal M(\Omega)} \big\lvert c_{\bm n}(\tilde f)\big\rvert.
    \end{equation*}
\end{thm}

\begin{proof}
    We apply \Cref{lem:bSeries} for $b=1>2d/\big(d+2\pare{k+\alpha}\big)$ and obtain that the series
    $\sum_{\bm{n} \in \Z^d} \lvert c_{\bm{n}}(\tilde{f})\rvert$ is convergent. Thus, we conclude by \cite[thm.~4.7]{PloPotSteTas18} that the Fourier series $\mathrm F\tilde f$ converges to $\tilde f$ uniformly on $\T^d$.
    For $\bm{x} \in \T^d$ and $\Omega\subset\Z^d$, we obtain
    \begin{equation*}
      \big\lvert \tilde{f} (\bm{x})- \mathrm F_\Omega \tilde{f}(\bm{x})\big\rvert
      = \Big\lvert \sum_{\bm{n} \in \Z^d} c_{\bm{n}}(\tilde{f})\, \e^{\im \inn{\bm{n},\bm{x}}} -\sum_{\bm{n} \in \Omega}c_{\bm{n}}(\tilde{f})\, \e^{\im \inn{\bm{n},\bm{x}}}\Big\rvert
      \le \sum_{\bm{n} \in \Z^d\setminus \Omega} \big\lvert c_{\bm{n}}(\tilde{f})\, \e^{\im \inn{\bm{n},\bm{x}}}\big\rvert.
    \end{equation*}
    \Cref{thm:manifold_DFS} now directly yields the second statement.
\end{proof}

We prove explicit bounds on the speed of convergence for the special cases of rectangular and {circular partial Fourier sums} \cite[p.~7~f.]{KhaNik92}. The proof of the following technical lemma, which is based on \cite[Thm.~3.2.16]{Gra08}, is found in \Cref{app:proof}.

\begin{lem}\label{lem:diadic_bds}
	Let $k \in \NN$, $0<\alpha < 1$, and $g \in \mathcal{C}^{k,\alpha}(\T^d)$. Then, we have for all $\ell \in \NN$
  \begin{equation}\label{eq:diadic_bds}
  	\sum_{\bm n \in \Z^d, \, 2^\ell\le \abs{\bm n} < 2^{\ell+1}} \abs{c_{\bm{n}}(g)} \le 2^{d-\alpha} \, d^{k+\frac{3}{2}} \, \pi^\alpha \, 2^{\ell\pare{\frac{d}{2}-(k+\alpha)}} \, \abs{g}_{\mathcal{C}^{k,\alpha}(\T^d)}.
  \end{equation}
\end{lem}

\begin{thm} \label{thm:hoelSrate}
  Let $k\in \NN$, $0<\alpha \le 1$ such that $2\pare{k+\alpha}>d$ and let $f\in\mathcal{C}^{k,\alpha}(\M)$. We define the circular partial DFS Fourier sums $\mathrm K_h f \coloneqq \mathrm S_{\Omega_h}f, \, h \in \N$, associated with ${\Omega_h=\{\bm{n} \in \Omega_\phi \mid \abs{\bm{n}} \le h\}}$. It holds that
  \begin{equation} \label{eq:hoelSrate}
    \norm{f - \mathrm K_h f}_{\mathcal{C}(\M)} 
    \le M_{d,d^\prime \hspace*{-2.5pt},k,\alpha}\, \norm{f}_{\mathcal{C}^{k,\alpha}(\M)}\, h^{\frac{d}{2}-k-\alpha},
  \end{equation}
  where 
  \begin{equation}\label{eq:M_bound1}
  	M_{d,d^\prime\hspace*{-2.5pt},k,\alpha}\coloneqq \frac{2^{\frac{d}{2}+k+1-\lfloor\alpha\rfloor} \, d^{k+2} \, \pi^\alpha \,  (k+d^\prime)!}{(1-2^{\frac{d}{2}-k-\alpha})\, (d^\prime-1)!}
    \quad \text{for } k+d^\prime\ge 2
  \end{equation} 
  and
  $
  	M_{1,1,0,\alpha}\coloneqq {2^{\frac{1}{2}+\lfloor\alpha\rfloor} \, \pi^\alpha }/{(1-2^{\frac{1}{2}-\alpha})}
  $. Here, $\lfloor \cdot \rfloor$ denotes rounding down to an integer.
\end{thm}

\begin{proof}
	 By \Cref{thm:hoelScon}, we can bound the left-hand-side by the sum over the remaining Fourier coefficients. We observe that $\mathcal M(\Omega_h) = \{\bm n \in \Z^d \mid \abs{\bm n} \le h,\allowbreak {r_{\bm n,\bm n}\neq0}\}$. This yields
  \begin{equation*}
    \big\lVert f - \mathrm K_h f\big\rVert_{\mathcal{C}(\M)} 
    \le \sum_{\bm{n} \in \mathbb{Z}^d, \,\abs{\bm{n}} >h} 
    \big\lvert c_{\bm{n}}(\tilde{f})\big\rvert \le \sum_{\ell=\lfloor \log_2 h \rfloor}^\infty \,\sum_{\bm n \in \Z^d, \, 2^\ell \le \abs{\bm n} < 2^{\ell+1}} \big\lvert c_{\bm n}(\tilde f) \big\rvert.
  \end{equation*}
  If $\alpha<1$, we apply \Cref{thm:hC_bound} to get $\tilde f \in \mathcal{C}^{k,\alpha}(\T^d)$ and thus we obtain by \Cref{lem:diadic_bds} that
  \begin{equation*}
  	\big\lVert f - \mathrm K_h f \big\rVert_{\mathcal{C}(\M)}\le 2^{d-\alpha} \, d^{k+\frac{3}{2}} \, \pi^\alpha \, \big\lvert \tilde f \big\rvert_{\mathcal{C}^{k,\alpha}(\T^d)} \sum_{\ell = \lfloor \log_2 h \rfloor}^\infty 2^{\ell \pare{\frac{d}{2}-(k+\alpha)}}.
  \end{equation*}
  Since $d/{2}-(k+\alpha)<0$ and $2^{\lfloor \log_2 h \rfloor} \ge h/2$, we can evaluate the geometric sum
  \begin{equation*}
  	\sum_{\ell = \lfloor \log_2 h \rfloor}^\infty
  2^{\ell\pare{\frac{d}{2}-(k+\alpha)}}
  = \frac{2^{\lfloor \log_2 h \rfloor \pare{\frac{d}{2}-k-\alpha}}}{1-2^{\frac{d}{2}-k-\alpha}}
  \le \frac{2^{k+\alpha-\frac{d}{2}} \, h^{\frac{d}{2}-k-\alpha}}{1-2^{\frac{d}{2}-k-\alpha}}.
  \end{equation*}
 \Cref{thm:hH_bound} yields that, if $k+d^\prime\ge 2$, we have $\abs{\smash{\tilde f}}_{\mathcal{C}^{k,\alpha}(\T^d)}\le 2  \sqrt{d} \, \frac{(k+d^\prime)!}{(d^\prime-1)!} \, \norm{f}_{\mathcal{C}^{k,\alpha}(\M)}$, and, if $k+d^\prime=1$, we have $\abs{\smash{\tilde f}}_{\mathcal{C}^{0,\alpha}(\T^1)}\le \norm{f}_{\mathcal{C}^{0,\alpha}(\M)}$. Thus, we overall obtain \eqref{eq:hoelSrate}. 
  In the case of $\alpha=1$, we use \Cref{thm:hC_bound} to deduce $\tilde f \in \mathcal{C}^{k,1-\varepsilon}(\T^d)$ for any $0<\varepsilon<1$. Choosing $\varepsilon$ small enough, we have ${d}/{2}-k-1+\varepsilon<0$ and obtain, by the same arguments as before, that
  \begin{equation*}
  	\big\lVert f -\mathrm K_h f\big\rVert_{\mathcal{C}(\M)} \le 2^{d-(1-\varepsilon)} \, d^{k+\frac{3}{2}} \, \pi^{1-\varepsilon} \, \big\lvert \tilde f \big\rvert_{\mathcal{C}^{k,1-\varepsilon}(\T^d)} \, \frac{2^{k+1-\varepsilon-\frac{d}{2}} \, h^{\frac{d}{2}-k-1+\varepsilon}}{1-2^{\frac{d}{2}-k-1+\varepsilon}}. 
  \end{equation*}
  By \eqref{eq:hC_bound}, we have $\abs{\smash{\tilde{f}}}_{\mathcal{C}^{k,1-\varepsilon}(\T^d)}\le \sqrt{d} \, \frac{(k+d^\prime)!}{(d^\prime-1)!} \, \norm{f}_{\mathcal{C}^{k,1}(\M)}$ if $k+d^\prime\ge 2$ and if $k+d^\prime=1$, we have $\abs{\smash{\tilde f}}_{\mathcal{C}^{0,1-\varepsilon}(\T^1)}\le 2 \norm{f}_{\mathcal{C}^{0,1}(\M)}$. Since the rest of the expression on the right-hand-side of the last inequality is continuous in $\varepsilon$, we can pass to the limit $\varepsilon\to1$ and finally obtain \eqref{eq:hoelSrate} for $\alpha=1$.
\end{proof}

\begin{rem}\label{rem:rect}
	An analogue to \Cref{thm:hoelSrate} still holds for the \emph{rectangular partial Fourier sums }$\mathrm S_{\Omega_h} f$, $h \in \N$, associated with $\Omega_h=\{\bm{n}\in \Omega_\phi \mid \norm{\bm n}_{\infty}<h\}$. In \Cref{app:proof}, we show that in this situation, the bound \eqref{eq:diadic_bds} and therefore the constant in \eqref{eq:M_bound1} for $k+d^\prime\ge 2$ can be improved to
	\begin{equation*}
		M_{d,d',k,\alpha}^{\mathrm{rect}} = \frac{2^{\frac{d}{2}+k+1-\lfloor\alpha\rfloor} \, d \, \pi^\alpha \,  (k+d^\prime)!}{(1-2^{\frac{d}{2}-k-\alpha})\, (d^\prime-1)!}.
	\end{equation*}
\end{rem}

%% file: 06_examples.tex
\section{The DFS method for specific manifolds}\label{sec:examples}
We present some application examples of the general DFS method. Furthermore, the DFS methods on the disk, ball, and cylinder from the literature are reviewed in the context of our generalized description, where we can now state the DFS basis functions.

\subsection{The interval}
A ``toy example'' of our general framework is the following generalized DFS method on the one-dimensional manifold $\M=[-1,1] \subset \R$.  We can easily show that
\begin{equation}\label{eq:dfs_interval}
  \phi_{[-1,1]} \colon \T^1 \to [-1,1], \, x \mapsto \cos x
\end{equation}
is a generalized DFS transform of $\M$ with symmetry number $p=1$ and symmetry function $s^1(x)= -x$, $x \in \T^1.$ Restricted to $D_1\coloneqq [0,\pi]$, the map $\phi_{[-1,1]}$ is bijective with its continuous inverse given by the arccosine. The derivative $\nabla\phi_{[-1,1]}(x) = \sin(x)$ is non-zero in $D_1^\circ=(0,\pi)$. We set $D_2\coloneqq \emptyset$.

We have $\mathcal{M}(n)=\{n,-n\}$ and $\mathrm N^1(n)=0$ for all $n \in \Z$, thus we can set $\Omega_{[-1,1]}\coloneqq \NN$. For $n \in \NN$ and $\xi \in [-1,1]$, we obtain by \Cref{defin:F_manifold} that
\begin{equation*}
	b_n(\xi) =e_n\Big(\big(\restr{\phi_{[-1,1]}}{[0,\pi]}\big)^{-1}(\xi)\Big) =\begin{cases}\e^{\im n \arccos \xi}+\e^{-\im n \arccos \xi}=2 T_n(\xi), & n\not=0\\
	\e^{\im 0 \arccos \xi}=T_0(\xi), &n=0. \end{cases}
\end{equation*}
Here, $T_n$ denotes the $n$-th \emph{Chebyshev polynomial of the first kind} defined by 
\begin{equation*}
	T_n(\xi) \coloneqq \cos(n \arccos \xi), \qquad  \xi \in [-1,1].
\end{equation*}
As in \Cref{thm:weighted}, the transform $\phi_{[-1,1]}$ induces a weighted Hilbert space $\tilde{L}_2(-1,1)$, where the inner product of $f_1,f_2 \in \tilde{L}_2(-1,1)$ is given by
\begin{align*}
  \inn{f_1,f_2}_{\tilde{L}_2(-1,1)} & \coloneqq \frac{1}{\pi} \int_{-1}^1 \frac{f_1(\xi) \, \overline{f_2(\xi)}}{\sqrt{1-\xi^2}} \d \xi
  = \frac{1}{2\pi} \int_{\T^1} \tilde f_1(x) \overline{\tilde f_2(x)} \mathrm d x= \langle \tilde{f}_1,\tilde{f}_2 \rangle_{L_2(\T^1)}.
\end{align*}
The Chebyshev polynomials $T_n$, $n \in \NN$, form an orthogonal basis of this space. 
The $n$-th partial Chebyshev expansion of $f \in \tilde{L}(-1,1)$ coincides with the DFS Fourier sum $S_{\{0,...,n\}}f$ for all $n \in \NN$.

\begin{rem}
Our framework of generalized DFS methods of a $d$-dimensional manifold $\M$ relies on transforming a function from $\M$ to $\T^d$ and subsequently expanding it with a Fourier series. In contrast, some DFS-like methods discussed in the literature, 
namely for the unit disk \cite{WilTowWri17}, the cylinder \cite{ForTow20}, and the ball \cite{BouTow20}, use coverings where some variables are on the interval $[-1,1]$ instead of the torus $\T^1$. This also fits into our approach, since the connection between the one-dimensional Chebyshev series and a DFS Fourier series easily extends to the multivariate case, where we get mixed Fourier--Chebyshev series.
\end{rem}

\subsection{The hyperball}
\label{sec:hyperball}
We define the \emph{$d$-dimensional closed unit ball}
\begin{equation*}
	\B^d \coloneqq \{\bm u\in\R^d \mid \norm{\bm u} \le 1\}.
\end{equation*} 
Let $d\ge 2$. A parameterization of $\B^d$ is given by the \emph{spherical coordinates} 
\begin{equation*}
	\bm\xi(\rho, \bm \lambda)=\big(\xi_j(\rho, \bm \lambda)\big)_{j=1}^d, 
  \qquad 
	\xi_j(\rho, \bm \lambda ) = \begin{cases}\rho \cos(\lambda_j) \prod_{\ell=1}^{j-1} \sin(\lambda_\ell), & j<d\\
	\rho \prod_{\ell=1}^{d-1} \sin(\lambda_\ell), & j=d
	\end{cases}
\end{equation*}
for $\rho \in [0,1]$ and $\bm \lambda\in [0,\pi]^{d-2} \times [-\pi,\pi]$.
Extending the domain of the angular variables $\lambda_j$, $j \in [d-2]$, to $[-\pi,\pi]$ and substituting the radial variable $\rho$ by $\cos x_1$, we obtain a generalized DFS transform $\phi_{\B^d}$ of $\B^d$. For $j \in [d]$, the $j$-th component of $\phi_{\B^d}$ in $\bm x \in \T^d$ is given by
\begin{equation}\label{eq:ball_DFS}
    \big(\phi_{\B^d}(\bm{x})\big)_j \coloneqq \begin{cases}\cos(x_1)\cos(x_{j+1}) \prod_{\ell=2}^j \sin(x_\ell), & j<d\\
    \cos(x_1) \prod_{\ell=2}^d \sin(x_\ell), & j=d.\end{cases}
\end{equation}
We check the requirements of \Cref{defin:DFS} for $\phi_{\B^d}$.
The smoothness properties of $\phi_{\B^d}$ are obvious. The symmetry number is $p=d$ and symmetry functions are given by
\begin{equation}\label{eq:ball_sym}
\begin{alignedat}{2}
	& s^1(\bm x)=(-x_1,x_2,...,x_d), &&\qquad \bm x \in \T^d,\\ 
	& s^2(\bm x)=(x_1+\pi,x_2+\pi,x_3,...,x_d), &&\qquad \bm x \in \T^d,\\
	& s^i(\bm x)=(x_1,...,x_{i-2},-x_{i-1},x_i+\pi,x_{i+1},...,x_d), &&\qquad \bm x \in \T^d, \, 3\le i\le d. 
\end{alignedat}
\end{equation}
These symmetry functions together with the sets
\begin{align*}
	& D_1 = [0,\pi/2) \times (0,\pi)^{d-2} \times (-\pi,\pi],\\
	& D_2 = \bigcup_{j=2}^{d-1} \pare{[0,\pi/2) \times (0,\pi)^{j-2} \times \{0,\pi\} \times \{0\}^{d-j} } \cup\{\pi/2\} \times \{0\}^{d-1},
\end{align*}
understood as subsets of $\T^d$ via the canonical identification of points in $\R^d$ with their equivalence classes in $\T^d$, satisfy \ref{Prop1} to \ref{Prop4} in \Cref{defin:DFS}. To show \ref{Prop3}, we note that the arccosine is continuous and for $\bm x \in D_1$, we have $\cos(x_1)=\norm{\phi_{\B^d}(\bm x)}\not=0$ as well as $\sin(x_\ell)\not=0$ for all $\ell \in \{2,..,d-1\}$. Therefore, \eqref{eq:ball_DFS} yields for $\bm x \in D_1$ and $j\in [d-1]$ that
\begin{equation*}
	\cos(x_{j+1})=\frac{\big(\phi_{\B^d}(\bm x)\big)_j}{\cos(x_1)\prod_{\ell=2}^j\sin(x_\ell)}
  \quad\text{and} \quad
	\sin(x_d) = \frac{\big(\phi_{\B^d}(\bm x)\big)_d}{\cos(x_1)\prod_{\ell=2}^{d-1}\sin(x_\ell)},
\end{equation*}
which inductively implies that the inverse of $\restr{\phi_{\B^d}}{D_1}$ is continuous. Furthermore, we have 
$\lvert \det \nabla \phi_{\B^d}(\bm x)\rvert
= \sin (x_1) \cos^{d-1}(x_1) \prod^{d-1}_{j=2} \sin^{d-j}(x_j)\not=0$ for $\bm x \in D_1^\circ$, which is \ref{Prop4}.
For $\bm n \in \Z^d$ we have
\begin{equation*}
	\mathcal{M}(\bm n)=\{\bm m \in \Z^d \mid m_j=\pm n_j \text{ for } j<d \text{ and } m_d=n_d\},
\end{equation*}
$\mathrm N^1(\bm n)=0$, $\mathrm N^2(\bm n)=n_1+n_2$, and $\mathrm N^i(\bm n)=n_i$ for all $3\le i\le d$. Since the reflections $\mathrm M^i$ act on pairwise disjoint sets variables, we observe that $r_{\bm n,\bm n}=0$ if and only if $\mathrm N^i(\bm n)$ is odd for some $i \in [d]$ with $\mathrm M^i(\bm n)=\bm n$, i.e., if and only if $n_1+n_2$ is odd or $n_{i-1}=0$ and $n_i$ is odd for some $3 \le i \le d$. We define
\begin{equation*}
	\Omega_{\B^d} \coloneqq \mleft\{\bm n \in \NN^{d-1}\times \Z \;\middle|\;  \begin{alignedat}{1}& n_1+n_2 \text{ is even and } (n_{i-1}\not=0 \text{ or } n_i \text{ is even for all }3\le i\le d) \end{alignedat}\mright\}.
\end{equation*}
Let $\bm n \in \Omega_{\B^d}$. We set $I_{\bm n}\coloneqq \{i \in \{3,...,d\} \mid n_{i-1}\not=0\}$ and $J_{\bm n}\coloneqq I_{\bm n}$ if $n_1=0$ and $J_{\bm n} \coloneqq I_{\bm n}\cup \{1\}$ otherwise. Then the map
$
	I \mapsto \mathrm M^I(\bm n) 
$
is bijective from $J_{\bm n}$ into $\mathcal{M}(\bm n)$. By \eqref{eq:r_n,m}, we have $r_{\bm n,\mathrm M^I(\bm n)}=(-1)^{\mathrm N^I(\bm n)} \, r_{\bm n,\bm n}$ for all $I \subset [d]$. Since $\mathrm N^1(\bm n)=0$, this yields for $I \subset I_{\bm n}$ that
\begin{equation*}
r_{\bm n,\mathrm M^{I \cup \{1\}}(\bm n)}=r_{\bm n,\mathrm M^I(\bm n)}=(-1)^{\mathrm N^I(\bm n)}=\prod_{i \in I}(-1)^{n_i}.
\end{equation*}
Furthermore, for $\bm x \in \T^d$ and $j \in [d-1]$, it holds that
\begin{equation*}
	\e^{\im n_jx_j}+(-1)^{n_{j+1}}\e^{-\im n_jx_j} = \begin{cases} 2  \cos(n_jx_j), & n_{j+1} \text{ is even}\\
	 2 \im  \sin(n_jx_j), & n_{j+1} \text{ is odd.}\end{cases}
\end{equation*}
Overall, we obtain by \eqref{eq:en} the basis function
\begin{align*}
	e_{\bm n}(\bm x) &=\sum_{\bm m \in \mathcal{M}(\bm n)} r_{\bm n,\bm m} \, \e^{\im \inn{\bm m,\bm x}}=\sum_{I \subset J_{\bm n}} r_{\bm n,\mathrm M^I(\bm n)} \, \e^{\im \inn{\mathrm M^I(\bm n),\bm x}}\\
	&= \sum_{I \subset I_{\bm n}} \frac{\e^{\im n_1 x_1}+\e^{-\im n_1x_1}}{2^{1+\#I_{\bm n}-\#J_{\bm n}}} \, \Big(\prod_{i \in I_{\bm n} \setminus I} \e^{\im n_{i-1}x_{i-1}}\Big) \, \Big(\prod_{i \in I} (-1)^{n_i} \, \e^{-\im n_{i-1}x_{i-1}}\Big) \, \e^{\im n_dx_d} \\
	&= \frac{\e^{\im n_1x_1}+\e^{-\im n_1x_1}}{2^{1+\#I_{\bm n}-\#J_{\bm n}}}\, \Big( \prod_{\substack{2\leq j \leq d-1\\n_j\not=0}} \big( \e^{\im n_j x_j}+(-1)^{n_{j+1}} \, \e^{-\im n_j x_j}\big) \Big)\, \e^{\im n_d x_d}\\
	&= 2^{\#J_{\bm n}} \, \cos(n_1x_1) \Big(\prod_{\substack{2\le j \le d-1\\n_j\not=0\\n_{j+1} \text{ even}}} \cos(n_jx_j)\Big) \, \Big(\prod_{\substack{2\le j \le d-1\\n_j\not=0\\ n_{j+1} \text{ odd}}} \im \sin(n_jx_j)\Big) \, \e^{\im n_d x_d}.
\end{align*}
Let $\rho \in [0,1], \bm \lambda \in [0,\pi]^{d-2}\times [-\pi,\pi]$ such that $(\arccos(\rho),\bm \lambda) \in D_1 \cup D_2$. For $\bm n \in \Omega_{\B^d}$, we obtain from the definition of $b_{\bm n}$ in \eqref{eq:bn} that
\begin{equation*}
	b_{\bm n}(\bm \xi(\rho,\bm\lambda))=
			2^{\#\{j \in [d-1] \mid n_j\not=0\}} T_{n_1}(\rho)\Big(\smashoperator{\prod_{\substack{j\in[d-2]\\n_{j+1}\not=0\\n_{j+2} \text{ even}}}} \cos(n_{j+1}\lambda_j)\Big) \Big(\smashoperator{\prod_{\substack{j\in[d-2]\\n_{j+1}\not=0\\ n_{j+2} \text{ odd}}}} \im \sin(n_{j+1}\lambda_j)\Big) \, \e^{\im n_d \lambda_{d-1}}.
\end{equation*}

For the special cases of $d \in \{2,3\}$, DFS methods of $\B^d$ already exist in the literature. In the following, we discuss their connection to our general framework in more detail.

\paragraph{The disk}
The \emph{unit disk} $\D \coloneqq \B^2$ can be parameterized by the \emph{polar coordinates}
\begin{equation*}
	\bm\xi(\rho, \lambda) = (\rho \cos \lambda,\rho \sin \lambda), \qquad (\rho,\lambda) \in [0,1] \times \T^1.
\end{equation*}
The map $\bm \xi$ is $2\pi$-periodic in the angular variable $\lambda$, thus a series expansion of a function on the unit disk can be realized directly by a shifted Chebyshev--Fourier expansion in polar coordinates, cf.\ \cite[§~18.5]{Boy00}. In 1995, Fornberg \cite{For95} presented an alternative approach; to eliminate the boundary at $\rho=0$ via extending the domain of the radius $\rho$ to $[-1,1]$. The resulting \emph{extended polar coordinates} act on $[-1,1] \times \T^1$ and cover the disk twice.
A DFS-like method of the disk consists of transforming functions to the extended polar coordinates and subsequently expanding them via a Chebyshev--Fourier series, see \cite{WilTowWri17}. Substituting $(\rho,\lambda)=(\cos x_1,x_2)$ with $(x_1,x_2) \in \T^2$ yields the DFS transform $\phi_{\D}$ from \eqref{eq:ball_DFS}. This substitution changes the symmetry structure: $\phi_{\D}$ covers the disk four times instead of twice and satisfies 
\begin{equation*}
	\phi_{\D}(x_1,x_2)=\phi_{\D}(-x_1,x_2)=\phi_{\D}(x_1+\pi,x_2+\pi), \qquad (x_1,x_2) \in \T^2,
\end{equation*}
which agrees with the symmetry functions from \eqref{eq:ball_sym}. This is illustrated in \Cref{fig:disk}.
\begin{figure}[htb]
	\centering	
	\includegraphics[width=\textwidth,trim={40 20 40 20}, clip]{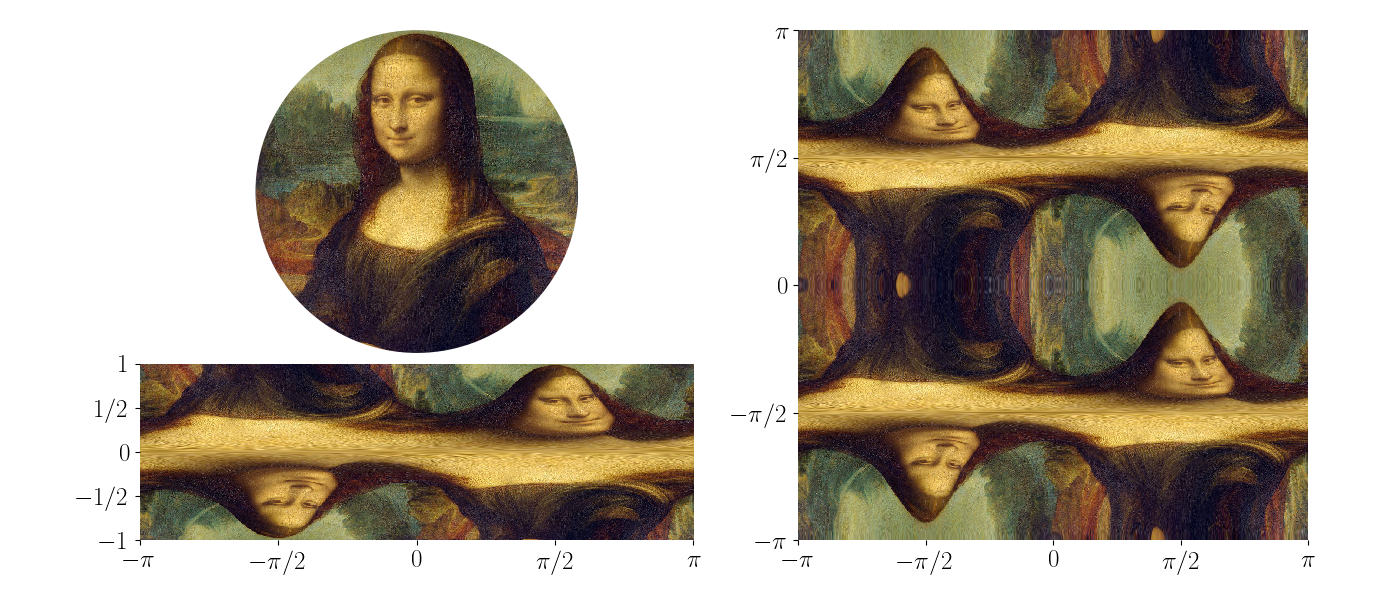}
  \caption{ Top left: A disk-shaped section $f(\bm\xi)$ of the Mona Lisa. Bottom left: $f(\bm\xi(\rho,\lambda))$ in extended polar coordinates. Right: DFS function $f\circ\phi_D(\bm x)$. \label{fig:disk}}
\end{figure}

Let $\bm n \in \Omega_{\D}\coloneqq\{\bm n \in \NN \times \Z \mid n_1+n_2 \text{ is even}\}$. For $(\rho,\lambda) \in (0,1] \times [-\pi,\pi] \cup \{(0,0)\}$, we have
\begin{equation*}
	b_{\bm n}(\bm \xi(\rho,\lambda))=
  \e^{\im n_2 \lambda} \,
	\begin{cases}
		1, &n_1=0\\
		2 \, T_{n_1}(\rho) , &n_1\not=0.
	\end{cases}
\end{equation*}

\paragraph{The ball}
The three-dimensional ball $\B^3$ is parameterized by the \emph{spherical coordinates}
\begin{equation*}
	\bm\xi(\rho, \theta, \lambda) = ( \rho \cos \theta, \rho \cos \lambda \sin \theta, \rho \sin \lambda \sin \theta), \qquad (\rho,\theta,\lambda)\in [0,1] \times [0,\pi]\times [-\pi,\pi].
\end{equation*}
Extending the domain of the radius $\rho$ to $[-1,1]$ and the polar angle $\theta$ to $[-\pi,\pi]$,
we obtain so-called \emph{extended spherical coordinates}.
Functions defined on the ball and represented in these extended spherical coordinates can then be expanded into Chebyshev--Fourier--Fourier series, cf.\ \cite{BouTow20}. 
The DFS transform $\phi_{\B^3}$ from \eqref{eq:ball_DFS} is obtained by substituting $(\rho,\theta,\lambda)=(\cos x_1,x_2,x_3)$ with $(x_1,x_2,x_3) \in \T^3$
thus covering the ball eight times. 

Let $(\rho,\theta,\lambda) \in (0,1] \times (0,\pi) \times [-\pi,\pi]$, or $(\rho,\theta,\lambda)=(0,0,0)$, or $\rho \in (0,1]$, $\theta \in \{0,\pi\}$ and $\lambda=0$. For $\bm n \in 	\Omega_{{\B^3}} = \{\bm n \in \NN^2 \times \Z \mid n_1+n_2 \text{ is even and } n_3 \text{ is even if } n_2=0\}$, we have
\begin{equation*}
	b_{\bm n}(\bm \xi(\rho,\theta,\lambda))=
  2^{\#\{j \in [2]\mid n_j\not=0\}}\, T_{n_1}(\rho)\, \e^{\im n_3 \lambda}\,
	\begin{cases}
		\hphantom{\i}  \cos(n_2\theta) ,& n_3 \text{ is even}\\
		\im \,  \sin(n_2\theta) ,	& n_3 \text{ is odd}.
	\end{cases}
\end{equation*}

\subsection{The hypersphere}
We define the $d$-dimensional unit sphere 
\begin{equation*}
	\S^d \coloneqq \{\bm u \in \R^{d+1} \mid \norm{\bm u}=1\},
\end{equation*}
which is a smooth submanifold of $\R^{d'}$ with $d'=d+1$.
This subset of the ball $\B^{d+1}$ can be parameterized by restricting the $(d+1)$-dimensional spherical coordinates to $\rho=1$.  A DFS transform of $\S^d$ is obtained by restricting the DFS transform $\phi_{\B^{d+1}}$ from \eqref{eq:ball_DFS} to $x_1=\arccos (1)=0$, i.e.,
\begin{equation*}
	\phi_{\S^d} \colon \T^d \to \C, \, \bm x \mapsto \phi_{\B^{d+1}}(0,\bm x).
\end{equation*}
One can easily show that the DFS smoothness and symmetry properties of $\phi_{\B^{d+1}}$ transfer to $\phi_{\S^d}$, hence $\phi_{\S^d}$ is a DFS transform of $\S^d$ with symmetry number $p=d-1$.
By a similar derivation as for the DFS method of $\B^{d+1}$ in \Cref{sec:hyperball},
we can set
\begin{equation*}
  \Omega_{\S^d} \coloneqq \{\bm n \in \NN^{d-1} \times \Z \mid n_i\not=0 \text{ or } n_{i+1} \text{ even, for all }i<d\}
\end{equation*} 
and we obtain for all $\bm n \in \Omega_{\S^d}$ and $\bm \lambda \in D_1 \cup D_2$ that
\begin{equation*}
	b_{\bm n}(\bm \xi(0,\bm \lambda))= 2^{\# \{j \in [d-1] \mid n_j\not=0\}} \Big(\smashoperator{\prod_{\substack{j \in [d-1]\\n_j\not=0\\n_{j+1} \text{ even}}}} \cos(n_j \lambda_j)\Big) \Big(\smashoperator{\prod_{\substack{j \in [d-1]\\n_j\not=0\\ n_{j+1} \text{ odd}}}} \im \sin(n_j\lambda_j)\Big) \, \e^{\im n_d \lambda_d},
\end{equation*}
where
\begin{equation*}
  D_1=(0,\pi)^{d-1} \times (-\pi,\pi],
  \qquad
  D_2= \bigcup_{j=1}^{d-1} \pare{(0,\pi)^{j-1}\times \{0,\pi\} \times \{0\}^{d-j}}.
\end{equation*}

\subsection{Product manifolds}\label{sec:product}

Let $\phi_1$ and $\phi_2$ be generalized DFS transforms of some $d_1$-dimensional manifold $\M_1 \subset \R^{d_1^\prime}$ and some $d_2$-dimensional manifold $\M_2 \subset \R^{d_2^\prime}$, respectively. It is not hard to verify that
\begin{equation*}
	\phi \colon \T^{d_1}\times \T^{d_2}  \to \M_1 \times \M_2, \, \bm x=(\bm x_1,\bm x_2)\mapsto \phi(\bm x)=(\phi_1(\bm x_1),\phi_2(\bm x_2))
\end{equation*}
defines a generalized DFS transform of the product manifold $\M_1 \times \M_2 \subset \R^{d_1^\prime+d_2^\prime}$. The smoothness properties are straightforward and the symmetry properties are as follows:
For $\ell \in \{1,2\}$, let $p_\ell$ denote the symmetry number, $s^i_\ell$, $i \in [p_\ell]$ the symmetry functions,
and $D_1^\ell$ and $D_2^\ell$ the subsets from \Cref{defin:DFS} 
with respect to $\phi_\ell$. 
Then, for $\phi$, we have the sets $D_1 \coloneqq D_1^1\times D_1^2$ and $D_2 \coloneqq (D_1^1 \times D_2^2) \cup (D_2^1 \times D_1^2) \cup (D_2^1\times D_2^2)$, the symmetry number $p_1+p_2$, and symmetry functions $s^i$, $i \in [p_1+p_2]$, can be chosen as
\begin{equation*}
	s^i \colon \T^{d_1} \times \T^{d_2} \to  \T^{d_1} \times \T^{d_2}, \, (\bm x_1,\bm x_2) \mapsto \begin{cases}
		\big(s_1^i(\bm x_1),\bm x_2\big), & i \le p_1\\
		\big(\bm x_1,s_2^{i-p_1}(\bm x_2)\big), & i>p_1.
	\end{cases}
\end{equation*}
Setting $\Omega_{\phi}\coloneqq \Omega_{\phi_1}\times \Omega_{\phi_2}$, simple calculations yield for all $\bm n=(\bm n_1,\bm n_2) \in \Omega_{\phi}$ and that
\begin{equation*}
	b_{\bm n}(\bm \xi)
= b_{\bm n_1}(\bm \xi_1) \, b_{\bm n_2}(\bm \xi_2)
,\qquad \bm \xi=(\bm \xi_1,\bm \xi_2) \in \M_1 \times \M_2.
\end{equation*}

\paragraph{The cylinder} An application of this product structure is the \emph{cylinder} $C\coloneqq\D\times[-1,1]$, 
where a DFS-like method was already derived in \cite{ForTow20}.
Combining our results on the disk $\D$ and the interval ${[-1,1]}$ from the previous subsections, we see that 
\begin{equation*}
	\phi_{C} \colon \T^2 \times \T^1 \to C, \, (\bm x_1,x_2) \mapsto (\phi_{\D}(\bm x_1),\phi_{[-1,1]}(x_2))
\end{equation*}
is a DFS transform of $C$ with symmetry number $3$. For $(x_1,x_2,x_3) \in \T^3\cong\T^2 \times \T^1$, the product symmetry structure yields
\begin{equation*}
	\phi_C(x_1,x_2,x_3)=\phi_C(-x_1,x_2,x_3)=\phi_C(x_1+\pi,x_2+\pi,x_3)=\phi_C(x_1,x_2,-x_3).
\end{equation*}
Let $\bm n \in \Omega_{C}\coloneqq \Omega_{\D} \times \NN= \{\bm n \in \NN \times \Z \times \NN\mid n_1+n_2 \text{ is even}\}$. In cylindrical coordinates, we have for all $(\rho,\lambda,z) \in (0,1] \times [-\pi,\pi] \times [-1,1]$ as well as for $(\rho,\lambda)=(0,0)$ and $z \in [-1,1]$ that
\begin{equation*}
	b_{\bm n}(\rho \cos \lambda, \rho \sin \lambda,z) =  2^{\#\{j \in \{1,3\} \mid n_j\not=0\}} \,  T_{n_1}(\rho) \, \e^{\im n_2 \lambda} \, T_{n_3}(z).
\end{equation*}
	
\subsection{The rotation group}
The three-dimensional \emph{rotation group}
\begin{equation*}
	\SO \coloneqq \{A \in \R^{3 \times 3} \mid \det(A)=1 \text{ and } A^{-1}=A^\intercal\}
\end{equation*}	
can be parametrized by the \emph{Euler angles}
\begin{equation*}
  (\alpha,\beta,\gamma)\mapsto R_z(\alpha)  R_y(\beta)  R_z(\gamma), \qquad (\alpha,\beta,\gamma)\in \T^1\times [0,\pi] \times \T^1,
\end{equation*}
where
\begin{equation*}
    R_z(\theta) \coloneqq 
    \begin{pmatrix} 
    	\cos \theta & -\sin \theta & 0 \\ 
    	\sin \theta & \cos \theta & 0 \\ 
    	0 & 0 & 1
    \end{pmatrix}, \qquad R_y(\theta) \coloneqq
    \begin{pmatrix} 
    	\cos \theta & 0 & \sin \theta \\ 
    	0 & 1 & 0 \\ 
    	-\sin \theta & 0 & \cos\theta
    \end{pmatrix}, \qquad \theta \in \T^1
\end{equation*}
describe rotations around the $z$-axis and the $y$-axis, respectively. Different conventions regarding the choice and order of rotational axes exist, here we employ the $zyz$-convention, cf.\ \cite[p.~4]{PotPreVol09}. 
The authors are not aware of any DFS method being used in the literature to approximate functions on $\SO$. However, extending the domain of the second Euler angle $\beta$ yields a generalized DFS transform of $\SO$ that fits within our framework. Because \Cref{defin:DFS} requires a submanifold of a Euclidean space, we employ the canonical column-wise embedding
\begin{align*}
  \R^{3,3} \ni
	A=(a_{i j})_{i,j=1}^3 \mapsto (a_{1 1},a_{2 1}, a_{3 1}, a_{1 2}, a_{2 2}, a_{3 2}, a_{1 3}, a_{2 3}, a_{3 3})
  \in \R^9,
\end{align*}
which is an isometric isomorphism from $\R^{3 \times 3}$ equipped with the Frobenius norm to $(\R^9, \norm{\cdot})$. We define a generalized DFS transform of $\SO$ by 
\begin{equation*}
	\phi_{\SO}(\bm x) \coloneqq
	R_z(x_1) R_y(x_2) R_z(x_3)
	\in \R^{3\times3} \cong \R^{9},
	\qquad \bm x \in \T^3.
\end{equation*}
We need to verify that $\phi_{\SO}$ satisfies the smoothness properties of a DFS transform. Explicit calculation of the matrix product $R_z(x_1) R_y(x_2) R_z(x_3)$ and component-wise taking of partial derivatives shows that for all $\bm \mu \in \NN^3$ and $\ell \in [9]$ the partial derivative $\mathrm D^{\bm \mu}\phi_{\ell}$ exists, where $\phi_\ell$ denotes the $\ell$-th component of $\phi_{\SO}$. Let $\bm x=(x_1,x_2,x_3) \in \T^3$. Up to translations and interchange of the variables, $\mathrm D^{\bm \mu}\phi_{\ell}$ is either a product of cosines, which is uniformly bounded by $1$, or of the form 
\begin{equation*}\label{eq:so_deriv}
		\mathrm D^{\bm{\mu}}\phi_l(\bm{x})=\cos(x_1) \cos(x_2) \cos(x_3) + \sin(x_1) \sin(x_3).
\end{equation*}
In the latter case, if $\sgn\big(\cos (x_1) \cos (x_3)\big)=\sgn\big(\sin (x_1) \sin (x_3)\big)$, we have by the triangle inequality and the angle subtraction formula that
	\begin{align*}
		\abs{\mathrm D^{\bm{\mu}}\phi_l(\bm{x})} & \le \abs{\cos(x_2)} \, \abs{\cos(x_1)\cos(x_3)}+ \abs{\sin(x_1) \sin(x_3)}\\
		& \le \abs{\cos(x_1)\cos(x_3)}+\abs{\sin(x_1)\sin(x_3)}\\
		& = \abs{\cos(x_1)\cos(x_3)+\sin(x_1)\sin(x_3)} = \abs{\cos(x_1-x_3)}\le 1.
	\end{align*}
	Otherwise, we apply the previous argument to $(-x_1,x_2+\pi,x_3)$ and obtain
	\begin{equation*}
		\abs{\mathrm D^{\bm{\mu}}\phi_l(x_1,x_2,x_3)}=\abs{-\mathrm D^{\bm{\mu}}\phi_l(-x_1,x_2+\pi,x_3)}\le 1.
	\end{equation*}
	This proves \eqref{eq:DFSinf}, i.e., $\phi_{\SO}$ satisfies the smoothness properties of a DFS transform.

One can also show that $\phi_{\SO}$ satisfies the symmetry properties of a DFS transform with symmetry number $p=1$, symmetry function
\begin{equation*}
	s^1(x_1,x_2,x_3) = (x_1+\pi,-x_2,x_3+\pi), \qquad \bm x=(x_1,x_2,x_3) \in \T^3,
\end{equation*}
and the subsets
$
	D_1 \coloneqq (-\pi,\pi] \times (0,\pi) \times (-\pi,\pi]
$
and
$
	D_2 \coloneqq (-\pi,\pi] \times \{0,\pi\} \times \{0\} .
$

For $\bm n\in\Z^3$, we have $\mathrm N^1(\bm n) = n_1 + n_3$
and 
$\mathcal M(\bm n) = \{\bm n, (n_1,-n_2,n_3)\}$.
We set 
\begin{equation*}
	\Omega_{\SO}\coloneqq \{\bm n \in \Z \times \NN \times \Z \mid n_2\not=0 \text{ or }n_1+n_3 \text{ even}\}.
\end{equation*} 
Let $(\alpha,\beta,\gamma) \in D_1 \cup D_2$. Then, the function $b_{\bm n}$, $\bm n \in \Omega_{\SO}$, is given by
\begin{equation*}
  b_{\bm n}\big(R_z(\alpha) R_y(\beta) R_z(\gamma)\big)=
  \e^{\im n_1 \alpha} \, \e^{\im n_3\gamma}
  \begin{cases}
    1, & n_2=0\\
    2 \cos(n_2\beta), & n_2\not=0, \, n_1+n_3 \text{ even}\\
    2 \im  \sin(n_2\beta) , & n_2\not=0, \,  n_1+n_3 \text{ odd}.
  \end{cases}
\end{equation*}

%% file: acknowledgments.tex
\subsection*{Acknowledgments}

We would like to thank Alex Townsend for suggesting to investigate the convergence of DFS methods on different manifolds.
We gratefully acknowledge the funding by the DFG (STE 571/19-1,
project number 495365311) within SFB F68 (“Tomography Across the Scales”)
as well as by the BMBF under the project “VI-Screen” (13N15754).

%% file: appendix.tex
\appendix

\section{Appendix}\label{app:proof}

\begin{proof}[Proof \Cref{lem:faaDi}]
This is a generalization of the proof of \cite[lem.~4.1]{MilQue22}
to higher dimensions.
The claim of the lemma clearly holds for $\bm{\beta}=\bm{0}$ and all $k$. If $\bm{\beta}=\bm{e}^p$ for some $p \in [d]$, where $\bm{e}^p$ denotes the $p$-th unit vector, we can apply the chain rule to get
    \begin{equation*}
        \mathrm D^{\bm{\beta}}(g \circ h)= \mathrm D^{\bm{e}^p}(g \circ h)=\sum_{\ell=1}^d \pare{ \mathrm D^{\bm{e}^\ell}g \circ h} \,  \mathrm D^{\bm{e}^p}h_\ell,
    \end{equation*}
    so the claim holds for all $k\in\N$ with 
    \begin{equation*}
        n=d^\prime=\frac{d^\prime!}{(d^\prime-1)!}\leq \frac{(d^\prime+k-1)!}{(d^\prime-1)!}.
    \end{equation*}
    For $k > 1$ and $\abs{\bm{\beta}} >1$, we proceed inductively: Let $\bm{\beta}^+=\bm{\beta}+\bm{e}^p \in B^d_k$ for some $p \in [d]$ and $\bm{\beta}\in B^d_{k-1}$. Assume that
    \begin{equation*}
        \mathrm D^{\bm{\beta}}(g \circ h)=\sum_{i=1}^{n} \pare{\mathrm D^{\bm{\gamma}_i}g \circ h} \, \prod_{j=1}^{m_i} \mathrm D^{\bm{\mu}_{ij}}h_{\ell_{ij}},
    \end{equation*}
    for some constants that satisfy \eqref{eq:indH2} to \eqref{eq:indH6} for $k-1$.
    Since $g$ and $h$ are $k$-times continuously differentiable, their composition is also $k$-times continuously differentiable and, by the product rule, we have 
    \begin{align*}
        \mathrm D^{\bm{\beta}^+}(g \circ h) & =\mathrm D^{\bm{e}^p}\pare{\mathrm D^{\bm{\beta}}(g \circ h)}=\sum_{i=1}^n \mathrm D^{\bm{e}^p} \Big(\pare{\mathrm D^{\bm{\gamma}_i}g \circ h} \, \prod_{j=1}^{m_i} \mathrm D^{\bm{\mu}_{ij}}h_{\ell_{ij}}\Big)\\
        & = \underbrace{\sum_{i=1}^n \mathrm D^{\bm{e}^p} \pare{\mathrm D^{\bm{\gamma}_i}g \circ h}\, \prod_{j=1}^{m_i} \mathrm D^{\bm{\mu}_{ij}}h_{\ell_{ij}}}_{\eqqcolon A}+\underbrace{\sum_{i=1}^n \pare{\mathrm D^{\bm{\gamma}_i}g \circ h} \, \mathrm D^{\bm{e}^p} \Big(\prod_{j=1}^{m_i} \mathrm D^{\bm{\mu}_{ij}}h_{\ell_{ij}}\Big)}_{\eqqcolon B}.
    \end{align*}
    Since by our assumptions $\mathrm D^{\bm{\gamma}_i}g$ and $h$ are continuously differentiable for all $i$, we can apply the chain rule to each summand of $A$ and get
    \begin{equation*}
        A = \sum_{i=1}^n \mathrm D^{\bm{e}^p} \pare{\mathrm D^{\bm{\gamma}_i}g \circ h} \, \prod_{j=1}^{m_i} \mathrm D^{\bm{\mu}_{ij}}h_{\ell_{ij}} =\sum_{i=1}^n \sum_{\ell=1}^{d^\prime} \pare{\mathrm D^{\bm{e}^\ell+\bm{\gamma}_i}g \circ h} \, \mathrm D^{\bm{e}^p} h_\ell \, \prod_{j=1}^{m_i} \mathrm D^{\bm{\mu}_{ij}} h_{\ell_{ij}}.
    \end{equation*}
    Similarly, we know that $D^{\bm{\mu}_{ij}}h_{\ell_{ij}}$ is continuously differentiable for all $i$ and $j$ and thus, we can simplify $B$ by applying the product rule. This yields
    \begin{equation*}
        B = \sum_{i=1}^n \pare{\mathrm D^{\bm{\gamma}_i}g \circ h} \, \mathrm D^{\bm{e}^p}\Big(\prod_{j=1}^{m_i} \mathrm D^{\bm{\mu}_{ij}}h_{\ell_{ij}}\Big)= \sum_{i=1}^n \sum_{r=1}^{m_i} \pare{\mathrm D^{\bm{\gamma}_i}g \circ h} \, \mathrm D^{\bm{e}^p+\bm{\mu}_{ir}}h_{\ell_{ir}} \, \prod_{\substack{j=1\\j \not=r}}^{m_i} \mathrm D^{\bm{\mu}_{ij}}h_{\ell_{ij}}.
    \end{equation*}
    We combine the resulting sums and relabel the constants to obtain
    \begin{equation*}
        \mathrm D^{\bm{\beta}^+}(g \circ h)=\sum_{i=1}^{n^+} \pare{\mathrm D^{\bm{\gamma^+}_i}g\circ h} \, \prod_{j=1}^{m_i^+} \mathrm D^{\bm{\mu}_{ij}^+}h_{\ell_{ij}^+},
    \end{equation*}
    where for all $i^+ \in [n^+]$ and $j^+ \in [m_{i^+}^+]$ there exist $i \in [n]$, $j \in [m_i]$, and $\ell \in [d^\prime]$ such that
    \begin{align*}
        \bm{\gamma}_{i^+}^+=\bm{e}^\ell+\bm{\gamma}_i \text{ or } \bm{\gamma}_{i^+}^+=\bm{\gamma}_i & \underset{\eqref{eq:indH3}}{\implies} \abs{\bm{\gamma}_{i^+}^+} \leq k,\\
        m_{i^+}^+=1+m_i \text{ or } m_{i^+}^+=m_i & \underset{\eqref{eq:indH4}}{\implies} m_{i^+}^+\leq k,\\
        \bm{\mu}_{i^+j^+}^+=\bm{e}^p, \ \bm{\mu}_{i^+j^+}^+=\bm{\mu}_{ij}, \text{ or } \bm{\mu}_{i^+j^+}^+=\bm{e}^p+\mu_{ij} &\underset{\eqref{eq:indH5}}{\implies} \abs{\bm{\mu}_{i^+j^+}^+} \leq k,\\
        \ell_{i^+ j^+}^+=\ell_{ij} \text{ or } \ell_{i^+ j^+}^+=\ell & \underset{\eqref{eq:indH6}}{\implies} \ell_{i^+ j^+}^+ \in [d^\prime],
    \end{align*}
    and by \eqref{eq:indH2}, we have 
    \begin{equation*}
        n^+=\sum_{i=1}^n \Big(\sum_{\ell=1}^{d^\prime} 1+\sum_{r=1}^{m_i}1 \Big)=n\pare{d^\prime+\max_{i=1}^n m_i} \leq n \, (d^\prime+(k-1))\leq \frac{(d^\prime+k-1)!}{(d^\prime-1)!}.
    \end{equation*}
    Thus, we have found constants that satisfy \eqref{eq:indH1} to \eqref{eq:indH6} for $\bm{\beta}^+$ and $k$. 
\end{proof}

\begin{proof}[Proof of \Cref{lem:diadic_bds}]
	We closely follow the proof of \cite[thm.~3.2.16]{Gra08}. We replicate the proof here for the convenience of the reader as \cite{Gra08} does not explicitly provide the constant and also uses different conventions. In particular, they work with $1$-periodic functions and with the Euclidean norm of multi-indices, as opposed to $2\pi$-periodic functions and the $\ell_1$-norm as in our work; this makes some small adjustments in the proof necessary. By the Cauchy–Schwarz inequality, we obtain
	\begin{align*}
		\pare{\sum_{\bm n \in \Z^d, \,2^\ell\le \abs{\bm n} < 2^{\ell+1}}\abs{c_{\bm n}(g)}}^2 \le \pare{\sum_{\bm n \in \Z^d, \,2^\ell\le \abs{\bm n}< 2^{\ell+1}}1^2} \pare{\sum_{\bm n \in \Z^d, \,2^\ell\le \abs{\bm n}< 2^{\ell+1}}\abs{c_{\bm n}(g)}^2}.
	\end{align*} 
	We write $\norm{\bm n}_{\infty} \coloneqq \sup_{j\in [d]} \abs{n_j}$. For the first factor, it holds that
	\begin{equation*}
		\sum_{\substack{\bm n \in \Z^d, \, 2^\ell \le \abs{\bm n} < 2^{\ell+1}}} 1^2\le \#\{\bm{n}\in \Z^d \mid \abs{\bm{n}} < 2^{\ell+1}\} \le \# \{\bm{n} \in \Z^d \mid \norm{\bm{n}}_{\infty} < 2^{\ell+1}\}\le 2^{\ell d+2d}.
	\end{equation*}
	Clearly, the same bound holds when summing over $\{\bm n \in \Z^d \mid 2^\ell \le \norm{\bm n}_{\infty}<2^{\ell+1}\}$.
	
	Next, we show an estimate on the second factor. Let $\bm{n} \in \Z^d$ with $2^\ell \le \abs{\bm n} < 2^{\ell+1}$. We choose $j \in [d]$ such that $\abs{n_j}=\norm{\bm n}_{\infty}$. We thus have $0<\frac{2^\ell}{d} \le \frac{\abs{\bm{n}}}{d} \le \abs{n_j} \le \abs{\bm n} < 2^{\ell+1}$.
	 It is true that $\abs{\e^{\im t}-1}\ge \frac{2}{\pi}\abs{t}$ for all $-\pi\le t\le \pi$. Therefore, we obtain
	\begin{equation}\label{eq:exp_bound}
		\Big\lvert \e^{\im \inn{\bm n, \frac{\pi}{2^{\ell+1}}\bm{e}^j}}-1\Big\rvert=\Big\lvert\e^{\im \pi \frac{ n_j}{2^{\ell+1}}}-1\Big\rvert \ge \frac{2}{\pi} \, \abs{\frac{\pi \, n_j}{2^{\ell+1}}} \ge \frac{1}{d}.
	\end{equation}
	Now, we can bound the second factor as follows
	\begin{align*}
		& \sum_{\bm n \in \Z^d, \, 2^\ell\le \abs{\bm n} < 2^{\ell+1}} \abs{c_{\bm n}\pare{g}}^2 \le \sum_{j=1}^d \sum_{\substack{\bm n \in \Z^d, \, 2^\ell \le \abs{\bm n} < 2^{\ell+1}\\ \abs{n_j}=\norm{\bm n}_{\infty}}} \abs{c_{\bm n}\pare{g}}^2\\
		\centermathcell{\underset{\eqref{eq:exp_bound}}{\le}} & d^2 \, \sum_{j=1}^d \sum_{\substack{\bm n \in \Z^d, \, 2^\ell \le \abs{\bm n}< 2^{\ell+1}\\ \abs{n_j}=\norm{\bm n}_{\infty}}} \Big\lvert \e^{\im \inn{\bm n,\frac{\pi}{2^{\ell+1}}\bm{e}^j}}-1\Big\rvert^2 \, \abs{c_{\bm n}\pare{g}}^2 \, \frac{\abs{n_j}^{2k}}{\abs{n_j}^{2k}}.\\
		\intertext{The factor $d^2$ on the right-hand side can be omitted when substituting $\norm{\bm n}_{\infty}$ for $\abs{\bm n}$ everywhere, since then \eqref{eq:exp_bound} can correspondingly be estimated by $1$. Using the differentiation identity $c_{\bm n}({\mathrm D^{k{\bm e}^j}}g)=(\im \, n_j)^k \, c_{\bm{n}}(g)$, cf.\ \cite[p.~162]{PloPotSteTas18}, and the inequality $\frac{1}{\abs{n_j}}\le \frac{d}{2^\ell}$, we further estimate}
		\centermathcell{\le} & \frac{d^{2k+2}}{2^{2\ell k}} \, \sum_{j=1}^d \sum_{\substack{\bm n \in \Z^d, \, 2^\ell \le \abs{\bm n} < 2^{\ell+1}\\ \abs{n_j}=\norm{\bm n}_{\infty}}} \Big\lvert\e^{\im \inn{\bm n,\frac{\pi}{2^{\ell+1}}\bm{e}^j}}-1\Big\rvert^2 \, \Big\lvert c_{\bm n}\big(\mathrm D^{k\bm{e}^j}g\big)\Big\lvert^2.\\
		\intertext{If we replace $\abs{\bm n}$ by $\norm{\bm n}_{\infty}$, we can again eliminate the powers of $d$ in the estimate. Applying the translation identity $\e^{\im \inn{\bm{n},\bm x}} \, c_{\bm n} (g)=c_{\bm{n}}\big(g(\cdot + {\bm x})\big)$, $\bm x\in\R^d$, cf.\ \cite[p.~162]{PloPotSteTas18}, and the linearity of Fourier coefficients, we obtain}
		\centermathcell{\le} & \frac{d^{2k+2}}{2^{2 \ell k}} \, \sum_{j=1}^d \sum_{\bm n \in \Z^d} \abs{c_{\bm n}\pare{\mathrm D^{k\bm{e}^j}g \pare{\cdot +\frac{\pi}{2^{\ell+1}} \, \bm{e}^j}-\mathrm D^{k \bm{e}^j}g}}^2.\\
		\intertext{Finally, we get by the Parseval identity \cite[thm.~4.5]{PloPotSteTas18}}
		\centermathcell{=} & d^{2k+2} \, 2^{-2 \ell k} \, \sum_{j=1}^d \norm{\mathrm D^{k\bm{e}^j}g \pare{\cdot +\frac{\pi}{2^{\ell+1}} \, \bm{e}^j}-\mathrm D^{k \bm{e}^j}g}_{L_2(\T^d)}^2\\
		\centermathcell{\le} & d^{2k+2} \, 2^{-2 \ell k} \, d \, \sup_{\abs{\bm \beta}=k} \big\lvert \mathrm D^{\bm \beta}g\big\rvert^2_{\mathcal{C}^\alpha(\T^d)} \, \big\lvert \frac{\pi}{2^{\ell+1}}\big\rvert^{2\alpha}\\
	 \centermathcell{\le} & d^{2k+3} \, 2^{-2 \ell k-2\alpha(\ell+1)} \, \pi^{2\alpha} \, \abs{g}_{\mathcal{C}^{k,\alpha}(\T^d)}^2.
	\end{align*}
	Overall, this yields
	\begin{equation*}
		\sum_{\bm n \in \Z^d, \, 2^l \le \abs{\bm n}<2^{\ell+1}} \abs{c_{\bm n}\pare{g}}\le 2^{d-\alpha} \, d^{k+\frac{3}{2}} \, \pi^{\alpha}  \, 2^{\ell\pare{\frac{d}{2}-(k+\alpha)}}\abs{g}_{\mathcal{C}^{k,\alpha}(\T^d)}.
	\end{equation*}
	For rectangular dyadic sums in \Cref{rem:rect}, we analogously obtain
	\begin{equation*}\label{eq:diadic_rect}
		\sum_{\bm n \in \Z^d, \, 2^\ell \le \norm{\bm n}_{\infty}<2^{\ell+1}} \abs{c_{\bm n}\pare{g}}\le 2^{d-\alpha} \, d^{\frac{1}{2}} \, \pi^{\alpha} \, 2^{\ell\pare{\frac{d}{2}-(k+\alpha)}} \abs{g}_{\mathcal{C}^{k,\alpha}(\T^d)}.\qedhere
	\end{equation*}
\end{proof}